\documentclass[a4paper]{amsart}
\usepackage{amsthm,amsfonts,amsmath,amssymb}
\usepackage[abs]{overpic}
\usepackage{here}

\newtheorem{theorem}{Theorem}[section]
\newtheorem{lemma}[theorem]{Lemma}

\newtheorem{proposition}[theorem]{Proposition}
\newtheorem{corollary}[theorem]{Corollary}

\theoremstyle{definition}
\newtheorem{definition}[theorem]{Definition}
\newtheorem{example}[theorem]{Example}

\newtheorem{remark}[theorem]{Remark}

\theoremstyle{remark}
\newtheorem*{acknowledgements}{Acknowledgements}

\newcommand{\e}{\varepsilon}

\makeatletter

\@addtoreset{figure}{section}
\makeatother

\makeatletter
  
  \@addtoreset{equation}{section}
\makeatother

\begin{document}
\title[Classification of $2$-component virtual links up to $\Xi$-moves]{Classification of $2$-component virtual links up to $\Xi$-moves}

% first author information
\author[Jean-Baptiste Meilhan]{Jean-Baptiste Meilhan}
\address{Univ. Grenoble Alpes, CNRS, IF, 38000 Grenoble, France}
\email{jean-baptiste.meilhan@univ-grenoble-alpes.fr}

% second author information
\author[Shin Satoh]{Shin Satoh}
\address{Department of Mathematics, Kobe University, Rokkodai-cho 1-1, Nada-ku, Kobe 657-8501, Japan}
\email{shin@math.kobe-u.ac.jp}

% third author information
\author[Kodai Wada]{Kodai Wada}
\address{Department of Mathematics, Kobe University, Rokkodai-cho 1-1, Nada-ku, Kobe 657-8501, Japan}
\email{wada@math.kobe-u.ac.jp}

% mathesubject
\makeatletter
\@namedef{subjclassname@2020}{%
  \textup{2020} Mathematics Subject Classification}
\makeatother
\subjclass[2020]{Primary 57K12; Secondary 57K10}

% keywords
\keywords{virtual link, $\Xi$-move, odd writhe, linking class, Gauss diagram}

% grant 
%\thanks{}

%\date{\today}

% dedicate
%\dedicatory{}

%%%%%%%%%% abstract %%%%%%%%%%
\begin{abstract}
The $\Xi$-move is a local move generated by 
forbidden moves in virtual knot theory.
This move was introduced by Taniguchi and the second author, 
who showed that it characterizes the odd writhe of virtual knots, which is a fundamental invariant defined by Kauffman. In this paper, we extend this result by classifying $2$-component virtual links up to $\Xi$-moves, using refinements of the odd writhe and linking numbers.
\end{abstract}

\maketitle

%%%%%%%%%% Introduction %%%%%%%%%%
\section{Introduction}

Virtual knot theory developed by Kauffman in~\cite{K99} is a diagrammatic extension of the classical study of knots in $3$-space.
A \emph{virtual knot} is a generalized knot diagram, where one allows both classical and virtual crossings, regarded up to an extended set of Reidemeister moves. 
Alternatively, virtual knots can be described in terms of \emph{Gauss diagrams}, which are copies of $S^1$ endowed with signed and oriented chords, 
modulo certain deformations \cite{GPV}.
The set-theoretic inclusion of usual knot diagrams into virtual knot diagrams induces an injection of classical knots into virtual knots.

In virtual knot theory, we cannot pass a strand \lq over\rq\, or \lq under\rq\, a virtual crossing. 
These operations are called the {\em forbidden moves}. 
At the Gauss diagram level, forbidden moves allow to exchange the
relative positions of any two consecutive endpoints $P_{1}$ and $P_{2}$ of chords on a
circle. 
See the left of Figure~\ref{fused}, 
where $\e_{i}$ are the signs of chords.  
It is known that any virtual knot is deformed into the trivial knot by forbidden moves \cite{Kan,Nel}. 
Generally, the $\mu$-component virtual links $L=K_1\cup\cdots\cup K_{\mu}$ up to forbidden moves are classified by the $(i,j)$-linking numbers ${\rm Lk}(K_i,K_j)$ $(1\leq i\ne j\leq \mu)$ completely~\cite{ABMW,Nas,Oka}.

\begin{figure}[htbp]
\centering
    \begin{overpic}[width=11cm]{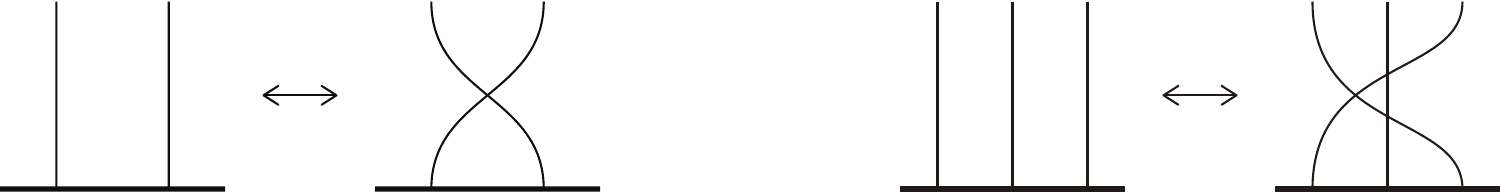}
       \put(30,-28){forbidden move}
       \put(0,31){$\e_{1}$}
       \put(39,31){$\e_{2}$}
       \put(79,31){$\e_{1}$}
       \put(116,31){$\e_{2}$}
       \put(6,-12){$P_{1}$}
       \put(30,-12){$P_{2}$}
       \put(85,-12){$P_{2}$}
       \put(109,-12){$P_{1}$}
       %%%%%
       \put(235,-28){$\Xi$-move}
       \put(184,31){$\e_{1}$}
       \put(200,31){$\e_{2}$}
       \put(216,31){$\e_{3}$}
       \put(264,31){$\e_{1}$}
       \put(279,31){$\e_{2}$}
       \put(307,31){$\e_{3}$}
       \put(190,-12){$P_{1}$}
       \put(206,-12){$P_{2}$}
       \put(222,-12){$P_{3}$}
       \put(269,-12){$P_{3}$}
       \put(285,-12){$P_{2}$}
       \put(301,-12){$P_{1}$}
    \end{overpic}
  \vspace{2em}
  \caption{The forbidden moves and $\Xi$-moves on Gauss diagrams} 
  \label{fused}
\end{figure}

The purpose of this paper is to study an operation called the {\it
$\Xi$-move}, which is generated by forbidden moves. 
At the Gauss diagram level, a $\Xi$-move swaps the positions of $P_{1}$ and $P_{3}$ of three consecutive endpoints $P_{1}$, $P_{2}$, and $P_{3}$ of chords. 
See the right of Figure~\ref{fused}. 
The $\Xi$-move arises naturally as a characterization of virtual knots having the same odd writhe. 
Here the \emph{odd writhe $J(K)$} is a fundamental invariant of a virtual knot $K$ in virtual knot theory defined by Kauffman in~\cite{K04}. 
In~\cite{ST}, Taniguchi and the second author proved the following. 

\begin{theorem}[{\cite[Theorem~1.7]{ST}}]
Let $K$ and $K'$ be virtual knots. 
Then the following are equivalent. 
\begin{enumerate}
\item[(i)] $K$ and $K'$ are related by a finite sequence of $\Xi$-moves. 
\item[(ii)] $J(K)=J(K')$. 
\end{enumerate}
\end{theorem}
 
We remark that Ohyama and Yoshikawa~\cite{OY} obtained the same result independently.

In this paper, we push further this study by classifying $2$-component virtual links up to $\Xi$-moves.
The situation turns out to be very different 
depending on the \emph{parity} of a virtual link. 
A $2$-component virtual link is called {\em odd} (resp.~{\em even}) if the number of classical crossings involving both components is odd (resp.~even) (Definition \ref{defparity}). 
Notice that the set of $2$-component even virtual links contains that of classical $2$-component links. 

In the odd case, we have the following. 

\begin{theorem}\label{th-odd}
Let $L=K_{1}\cup K_{2}$ and $L'=K'_{1}\cup K'_{2}$ be $2$-component odd virtual links.
Then the following are equivalent. 
\begin{enumerate}
\item[(i)] $L$ and $L'$ are related by a finite sequence of $\Xi$-moves. 
\item[(ii)] ${\rm Lk}(K_{1},K_{2})={\rm Lk}(K'_{1},K'_{2})$ and ${\rm Lk}(K_{2},K_{1})={\rm Lk}(K'_{2},K'_{1})$. 
\end{enumerate}
\end{theorem}

We remark that each component of a $2$-component odd virtual link can be unknotted by $\Xi$-moves. 
Furthermore, by the classification of $2$-component virtual links up to forbidden moves~\cite[Corollary 7]{Oka} (see also~\cite[Proposition 3.6]{ABMW}), Theorem~\ref{th-odd} means that the equivalence relation generated by $\Xi$-moves is coincident with that by forbidden moves on the set of $2$-component odd virtual links.

The even case is much less simple, and involves three new invariants. 
For a $2$-component even virtual link  $L=K_{1}\cup K_{2}$, we introduce the odd writhe of the pair $(L,K_{i})$ for $i=1,2$, denoted by $J(L,K_{i})$, 
as an extension of the original invariant defined in~\cite{K04} (Definition~\ref{def-oddwrithe}). 
Furthermore, we define the reduced linking class $\overline{F}(L)$ of $L$, 
which is a refinement of the linking numbers ${\rm Lk}(K_{1},K_{2})$ and ${\rm Lk}(K_{2},K_{1})$ (Definition~\ref{def-linkingclass}). 
Then we have the following. 

\begin{theorem}\label{th-even}
Let $L=K_{1}\cup K_{2}$ and $L'=K'_{1}\cup K'_{2}$ be $2$-component even virtual links. 
Then the following are equivalent. 
\begin{enumerate}
\item[(i)] $L$ and $L'$ are related by a finite sequence of $\Xi$-moves. 
\item[(ii)] $J(L,K_{1})=J(L', K'_{1})$, $J(L, K_{2})=J(L', K'_{2})$, and $\overline{F}(L)=\overline{F}(L')$. 
\end{enumerate}
\end{theorem}

This paper is organized as follows. 
In Section~\ref{sec-definitions}, we review the definitions of virtual links and Gauss diagrams. 
The proof of Theorem~\ref{th-odd} is given in Section~\ref{sec-odd} 
by showing that a forbidden move is generated by $\Xi$-moves 
in the case of $2$-component odd virtual links. 
In Section~\ref{sec-shell}, we study 
shells, which are a certain kind of self-chords 
in a Gauss diagram, and prove that any Gauss diagram can be deformed into 
a certain form with shells up to $\Xi$-moves 
(Proposition~\ref{prop-selfchord}). 
In Section~\ref{sec-standardform}, 
we introduce the notion of a ladder 
consisting of parallel nonself-chords in a Gauss diagram, 
and give a representative of the equivalence class of a $2$-component virtual link 
under $\Xi$-moves (Proposition~\ref{prop-even}). 
In Section~\ref{sec-even-inv}, we define three kinds of invariants 
$J(L, K_{1})$, $J(L, K_{2})$, and $\overline{F}(L)$ 
of a $2$-component even virtual link $L=K_{1}\cup K_{2}$, 
and establish a relationship among these invariants (Theorem~\ref{th-relation}). 
The last section is devoted to the proof of Theorem~\ref{th-even}.

\begin{acknowledgements}
The authors would like to thank Professors Takuji Nakamura and Yasutaka Nakanishi for useful comments and suggestions on an early version of the paper. 
They also thank the referee for the careful reading of the paper and for his/her comments and suggestions. 
The first author was partly supported by the project AlMaRe (ANR-19-CE40-0001-01) of the ANR.
The second author was partially supported by JSPS KAKENHI Grant Number JP19K03466. 
The third author was partially supported by JSPS KAKENHI Grant Number JP19J00006. 
\end{acknowledgements}

%%%%%%%%%% Virtual links and Gauss diagrams %%%%%%%%%%
\section{Virtual links and Gauss diagrams}\label{sec-definitions}
For an integer $\mu\geq1$, a {\em $\mu$-component virtual link diagram} is 
the image of an immersion of $\mu$ oriented and ordered  
circles into the plane, 
whose singularities are only transverse double points. 
Such double points are divided into {\em classical crossings} and {\em virtual crossings}. 
At a classical crossing, we distinguish two intersecting arcs, 
called the {\it over-arc} and {\it under-arc} formally, 
by removing a small path from the under-arc 
to indicate the crossing information. 
We also define the {\it sign} of a classical crossing 
with respect to the orientation of the arcs 
as shown in the left of Figure~\ref{xing}. 
At a virtual crossing, we surround it by a small circle. 
See the right of the figure.

\begin{figure}[htbp]
  \centering
    \begin{overpic}[width=7.5cm]{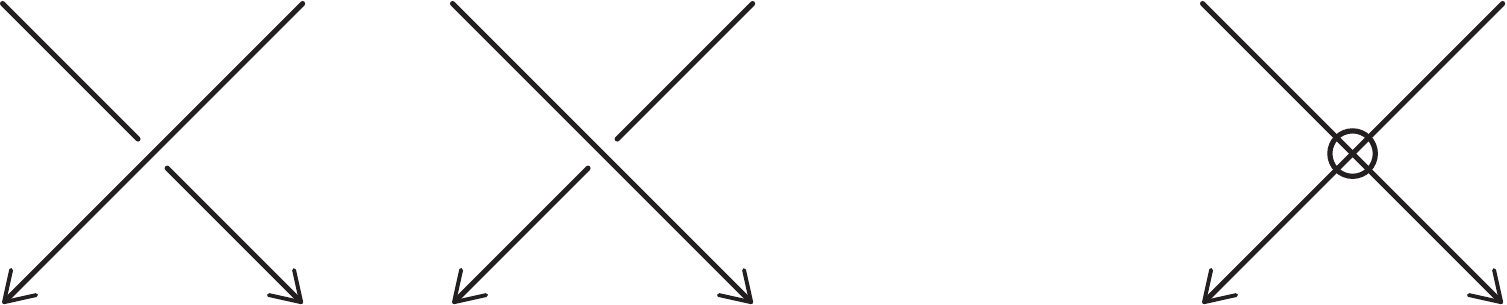}
      \put(18,3){$+$}
      \put(82,3){$-$}
      \put(-23,-15){positive/negative classical crossing}
      \put(159,-15){virtual crossing}
    \end{overpic}
  
  \vspace{1em}
  \caption{Two types of crossings}
  \label{xing}
\end{figure}

A {\em $\mu$-component virtual link} is an equivalence class of $\mu$-component virtual link diagrams under seven kinds of 
{\em generalized Reidemeister moves} 
R1--R3 and V1--V4 as shown in Figure~\ref{Rmoves} (cf.~\cite{K99}). 
In particular, 
the moves R1--R3 are called {\it classical Reidemeister moves}, 
and V1--V4 are {\it virtual Reidemeister moves}. 
We remark that the equivalence relation keeps 
the order of the components. 
For example, Figure~\ref{ex-sequence} shows a sequence of 
$2$-component virtual link diagrams related by 
generalized Reidemeister moves five times.

\begin{figure}[htbp]
\centering
    \begin{overpic}[width=12cm]{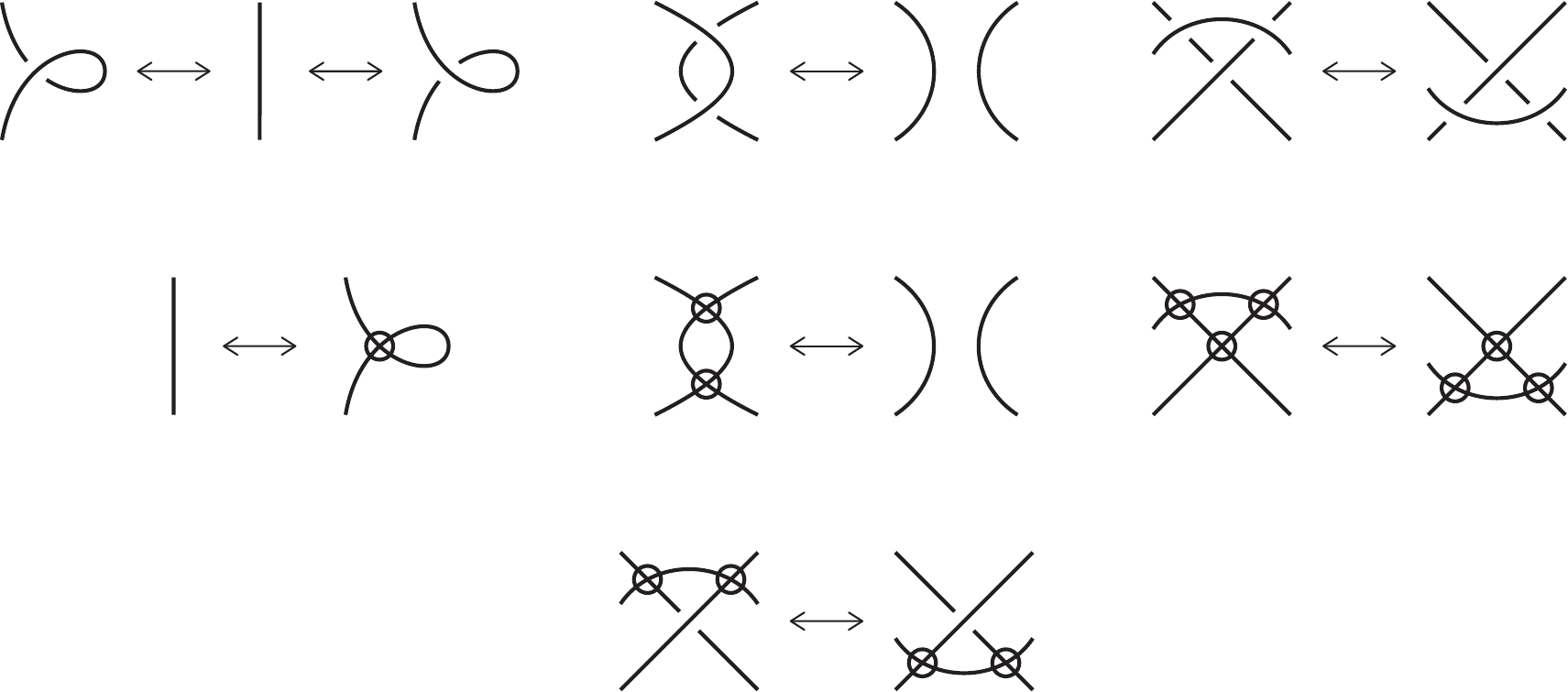}
      \put(32,140.5){R1}
      \put(69.5,140.5){R1}
      \put(174,140.5){R2}
      \put(290,140.5){R3}
      \put(50.5,80.5){V1}
      \put(174,80.5){V2} 
      \put(290,80.5){V3}
      \put(174,20.5){V4}
    \end{overpic}
  \caption{Generalized Reidemeister moves}
  \label{Rmoves}
\end{figure}

\begin{figure}[htbp]
\centering
    \begin{overpic}[width=12cm]{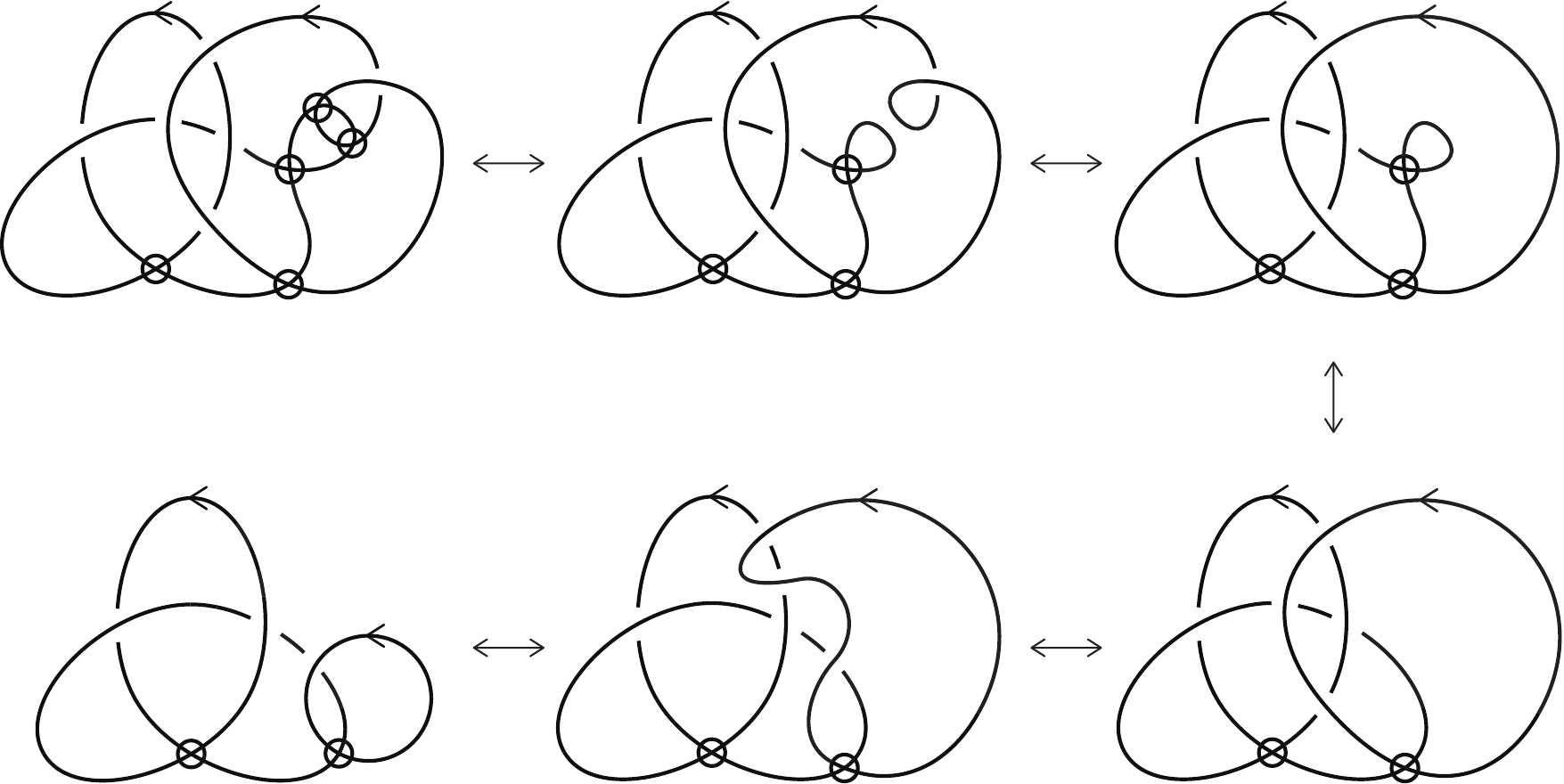}
      \put(105,140){V2}
      \put(227,140){R1}
      \put(297,80){V1}
      \put(227,35){R3}
      \put(105,35){R2}
    \end{overpic}
  \caption{A sequence of $2$-component virtual link diagrams}
  \label{ex-sequence}
\end{figure}

A {\em Gauss diagram} is a disjoint union of ordered and oriented circles together with signed and oriented chords whose endpoints lie disjointly on the circles. 
A chord in a Gauss diagram is called a {\em self-chord} if both endpoints 
of the chord lie on the same circle of the Gauss diagram; 
otherwise it is called a {\em nonself-chord}. 
We say that a self-chord $\gamma$ is {\em free} 
if the endpoints of $\gamma$ are adjacent on the circle; 
that is, one of the arcs on the circle spanned 
by the endpoints of $\gamma$ 
contains no endpoints of any other chords. 
Given a $\mu$-component virtual link diagram $D$ with $n$ classical crossings and some or no virtual crossings, 
the {\em Gauss diagram associated with $D$} 
is defined to be the union of $\mu$~circles 
corresponding to the components of $D$ 
and $n$~chords connecting the preimage of each classical crossing. 
Each chord of the Gauss diagram is equipped with the sign of the corresponding classical crossing, and oriented from the over-arc to the under-arc.

\begin{example}\label{ex21}
Let $L=K_1\cup K_2$ be a $2$-component 
virtual link presented by a virtual link diagram 
as shown in the left of Figure~\ref{ex-Gauss}. 
The corresponding Gauss diagram consists of two circles 
$C_1$ and $C_2$ 
with six chords, 
where $C_i$ corresponds to $K_i$ $(i=1,2)$. 
Three of the chords are self-chords 
corresponding to the crossings
labeled $1$, $2$, and $6$. 
In particular, the self-chord $6$ 
is a free chord. 
The other three chords labeled $3$, $4$, and $5$ are nonself-chords 
connecting the two circles $C_1$ and $C_2$. 
\end{example}

\begin{figure}[htbp]
\centering
    \begin{overpic}[width=10cm]{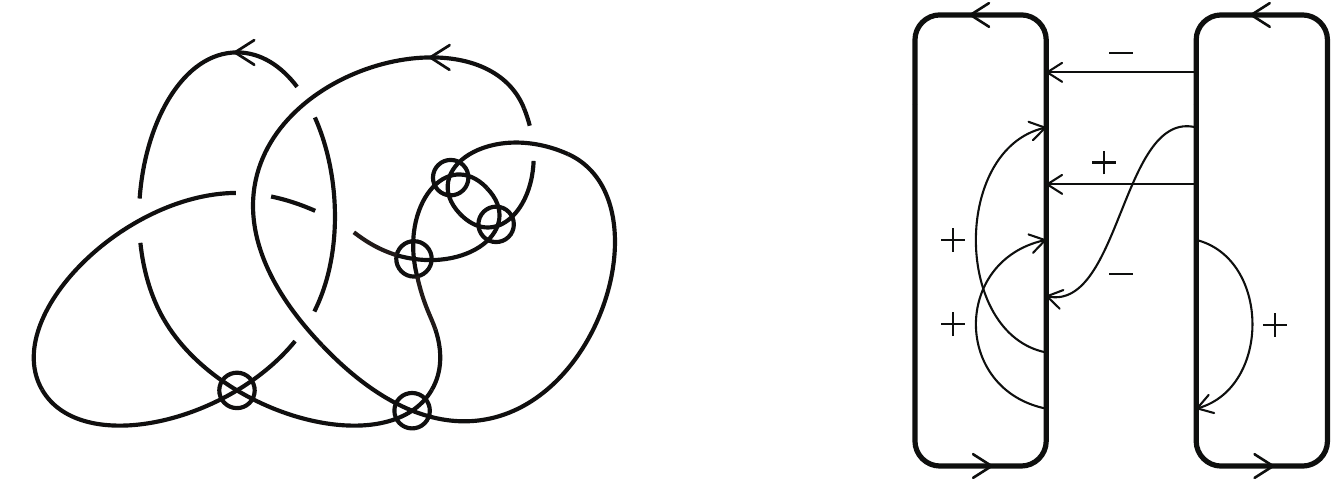}
      \put(-2,46){$K_{1}$}
      \put(136,46){$K_{2}$}
      \put(180,46){$C_{1}$}
      \put(289,46){$C_{2}$}
      %%%%%%%%%%%%%%%%
      \put(19,55){$1$}
      \put(76,56){$2$}
      \put(47,65){$3$}
      \put(64,18){$4$}
      \put(64.5,87){$5$}
      \put(118,74){$6$}
      %%%%%%%%%%%%%%%%%
      \put(228,47){$1$}
      \put(228,11){$1$}
      \put(228,71){$2$}
      \put(228,23){$2$}
      \put(259,71){$3$}
      \put(215,35){$3$}
      \put(215,59){$4$}
      \put(259,59){$4$}
      \put(215,84){$5$}
      \put(259,84){$5$}
      \put(247,47){$6$}
      \put(247,11){$6$}
    \end{overpic}
  \caption{A $2$-component virtual link diagram and its Gauss diagram}
  \label{ex-Gauss}
\end{figure}

We consider the deformations on Gauss diagrams 
corresponding to generalized Reidemeister moves on 
virtual link diagrams. 
By definition, the four virtual Reidemeister moves V1--V4 on virtual link diagrams do not affect the corresponding Gauss diagrams; 
in fact, a Gauss diagram has no information on virtual crossings. 
On the other hand, a classical Reidemeister move R1 
adds or removes a free chord in the corresponding Gauss diagram. 
See the left of Figure~\ref{R1R2-Gauss}, 
where $\varepsilon=\pm$. 
For a classical Reidemeister move R2, 
we have an addition or deletion of a pair of chords 
with the same direction and opposite signs 
such that their initial and terminal endpoints are adjacent, 
respectively. 
See the right of the figure.

\begin{figure}[htbp]
\centering
    \begin{overpic}[width=12cm]{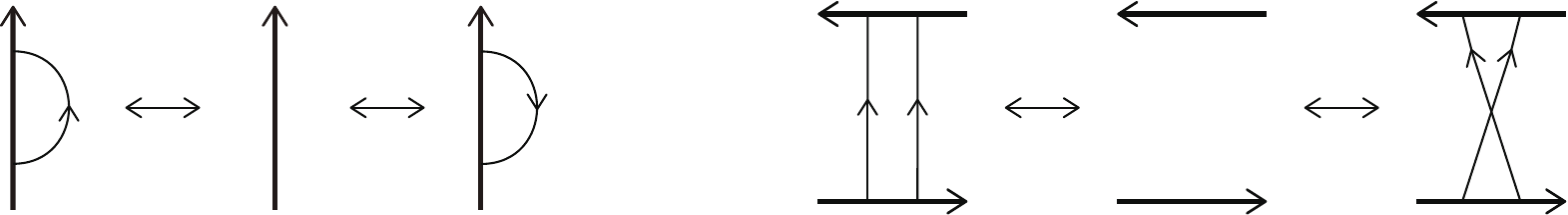}
      \put(30,29){R1}
      \put(78.5,29){R1}
      \put(221,29){R2}
      \put(286,29){R2}
      %%%
      \put(15,12){$\e$}
      \put(117,12){$\e$}
      %%%
      \put(180,12){$\e$}
      \put(203,12){$-\e$}
      \put(313,12){$\e$}
      \put(331,12){$-\e$}
    \end{overpic}
  \caption{Reidemeister moves R1 and R2 on Gauss diagrams}
  \label{R1R2-Gauss}
\end{figure}

A classical Reidemeister move R3 
involves three arcs and three classical crossings on a virtual link diagram. 
There are several types of R3's 
depending on the orientations of the arcs 
and the over/under information at the crossings. 
Figure~\ref{R3-Gauss} shows two typical examples of 
Reidemeister moves R3 
and the corresponding deformations on Gauss diagrams. 
We will use these two R3's in Section~\ref{sec-standardform}. 
Therefore a virtual link can be considered as an equivalence class of all Gauss diagrams under Reidemeister moves R1--R3 (cf. \cite{GPV,K99}).

\begin{figure}[htbp]
\centering
    \begin{overpic}[width=12cm]{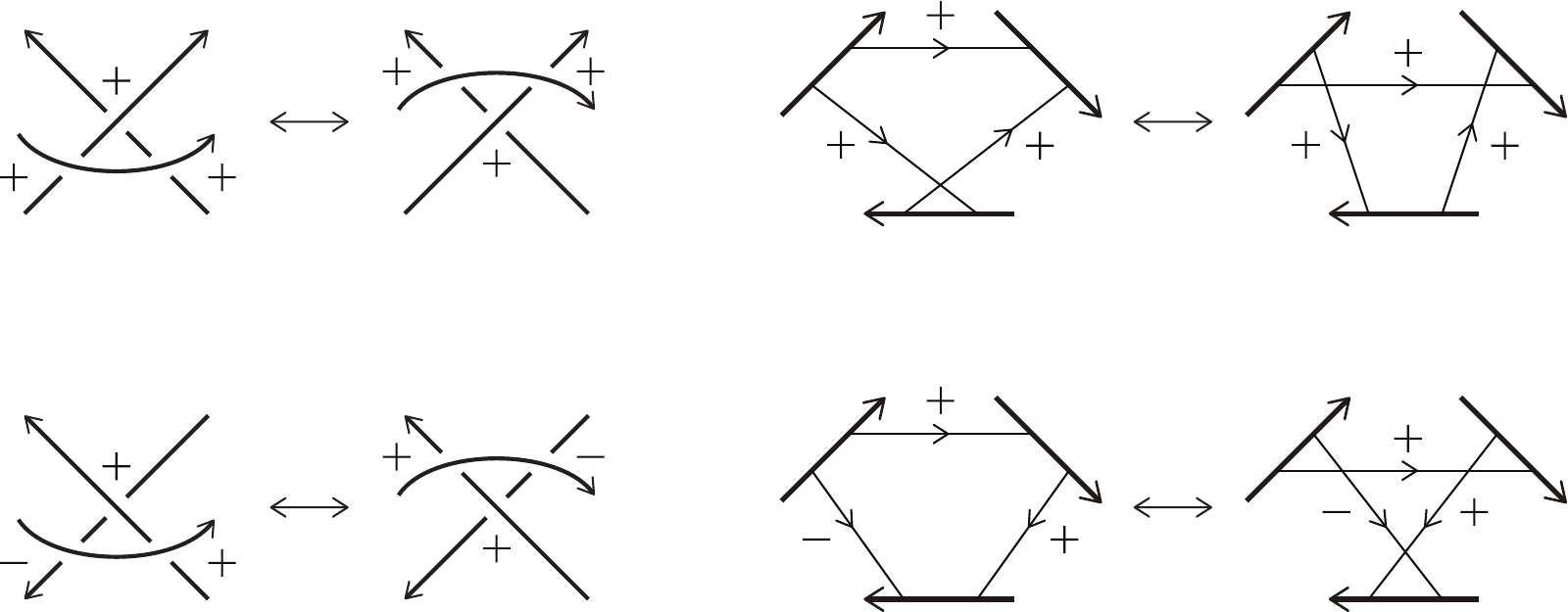}
      \put(61.25,112){R3}
      \put(249,112){R3}
      \put(61.25,28){R3}
      \put(249,28){R3}
    \end{overpic}
  \caption{Two Reidemeister moves R3}
  \label{R3-Gauss}
\end{figure}

Now we introduce another deformation 
on a Gauss diagram. 
Let $P_{1},P_{2}$, and $P_{3}$ be 
three consecutive endpoints  of chords 
lying on the same circle of a Gauss diagram. 
A {\em $\Xi$-move} \cite{ST} is 
a deformation which exchanges the positions of $P_{1}$ and $P_{3}$ 
with preserving the signs and orientations of the chords. 
See Figure~\ref{Xi-Gauss}. 
In the definition of a $\Xi$-move, 
we consider all signs and orientations of the chords, 
and some of the chords are possibly the same. 
In the figure, a pair of dots $\bullet$ marks 
the two endpoints $P_1$ and $P_3$ exchanged by a $\Xi$-move.

\begin{figure}[htbp]
\centering
    \begin{overpic}[width=5cm]{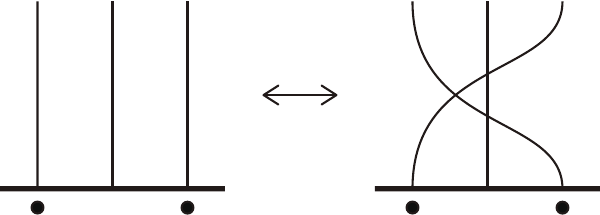}
      \put(4,-13){$P_{1}$}
      \put(22,-13){$P_{2}$}
      \put(40,-13){$P_{3}$}
      \put(-2,41){$\e_{1}$}
      \put(15.5,41){$\e_{2}$}
      \put(33.5,41){$\e_{3}$}
      %%%
      \put(92.5,-13){$P_{3}$}
      \put(110.5,-13){$P_{2}$}
      \put(128.5,-13){$P_{1}$}
      \put(88,41){$\e_{1}$}
      \put(105,41){$\e_{2}$}
      \put(134,41){$\e_{3}$}
    \end{overpic}
  \vspace{1em}
  \caption{A $\Xi$-move}
  \label{Xi-Gauss}
\end{figure}

We say that 
two Gauss diagrams $G$ and $G'$ are {\em $\Xi$-equivalent} if they are related by a finite sequence of Reidemeister moves R1--R3 and $\Xi$-moves. 
We denote it by $G\sim G'$. 
Two virtual links are {\em $\Xi$-equivalent} if their Gauss diagrams are $\Xi$-equivalent. 

\bigskip

In the rest of this paper, 
we only consider {\em $2$-component} virtual links. 
Let $L=K_{1}\cup K_{2}$ be a $2$-component virtual link, and $G$ a Gauss diagram 
presenting $L$ with two circles $C_{1}$ and $C_{2}$ corresponding to 
the components $K_1$ and $K_2$ of $L$, 
respectively. 
For integers $i$ and $j$ with $\{i,j\}=\{1,2\}$, 
a chord in $G$ is called {\em of type~$(i,j)$} if it is 
a nonself-chord connecting $C_i$ and $C_j$ 
and oriented from $C_{i}$ to $C_{j}$. 

\begin{definition}[{cf.~\cite[Section 1.7]{GPV}}]\label{def-Lk}
For $i$ and $j$ with $\{i,j\}=\{1,2\}$, 
the {\em $(i,j)$-linking number} of $L=K_1\cup K_2$ is the sum of signs of all nonself-chords of type~$(i,j)$ in $G$. 
We denote it by ${\rm Lk}(K_{i},K_{j})$. 
\end{definition}

The integers ${\rm Lk}(K_1,K_2)$ and ${\rm Lk}(K_2,K_1)$ are 
both invariants of the virtual link $L$; 
that is, they do not depend on a particular choice 
of a Gauss diagram $G$ presenting $L$.

\begin{definition}[cf.~\cite{MWY}]\label{defparity}
The {\it parity of a $2$-component virtual link} $L=K_1\cup K_2$ is 
the parity of ${\rm Lk}(K_1,K_2)+{\rm Lk}(K_2,K_1)$. 
\end{definition}

By definition, the parity of $L$ is coincident with 
that of the number of nonself-chords in any Gauss diagram of $L$. 
For example, the $2$-component virtual link 
$L=K_1\cup K_2$ given in Example~\ref{ex21} 
satisfies 
${\rm Lk}(K_1,K_2)=0$ and ${\rm Lk}(K_2,K_1)=-1$, 
and therefore, $L$ is odd. 

We have a relationship between a $\Xi$-move and  
linking numbers as follows. 

\begin{lemma}\label{lem-Lk}
The $(i,j)$-linking number ${\rm Lk}(K_{i},K_{j})$ 
of $L=K_1\cup K_2$ is invariant under $\Xi$-moves 
for $\{i,j\}=\{1,2\}$. 
Therefore the parity of $L$ is preserved under $\Xi$-moves. 
\end{lemma}

\begin{proof}
For any nonself-chord in a Gauss diagram, 
a $\Xi$-move does not change its sign and type $(i,j)$. 
Therefore the sum of signs of all nonself-chords of type~$(i,j)$ 
is preserved. 
\end{proof}

%%%%%%%%%% Odd case %%%%%%%%%
\section{Proof of Theorem~\ref{th-odd}}\label{sec-odd} 

In this section, 
we study $2$-component {\it odd} virtual links, 
and prove Theorem~\ref{th-odd} 
by giving a complete representative system of 
$\Xi$-equivalence classes of the links.

\begin{lemma}\label{lem-cross-move}
A deformation on a Gauss diagram 
as shown in {\rm Figure~\ref{cross-move}} 
is realized by two $\Xi$-moves. 
\end{lemma}

\begin{figure}[htbp]
\centering
    \begin{overpic}[width=6cm]{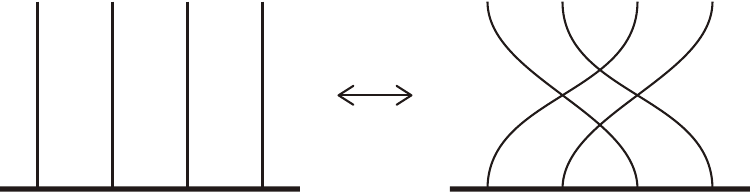}
    \end{overpic}
  \caption{A deformation in Lemma~\ref{lem-cross-move}}
  \label{cross-move}
\end{figure}

\begin{proof}
We apply a $\Xi$-move to the first and third endpoints, 
and another $\Xi$-move to the second and fourth endpoints. 
\end{proof}

\begin{lemma}\label{lem-odd}
A deformation on a Gauss diagram 
of a $2$-component odd virtual link 
as shown in {\rm Figure~\ref{ends-exchange}} is realized by 
a finite number of $\Xi$-moves. 
\end{lemma}

\begin{figure}[htbp]
\centering
    \begin{overpic}[width=5cm]{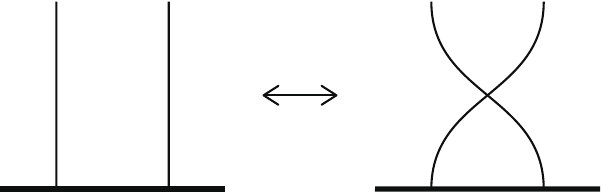}
    \end{overpic}
  \caption{A deformation in Lemma~\ref{lem-odd}}
  \label{ends-exchange}
\end{figure}

\begin{proof} 
We may prove this lemma for two consecutive endpoints on 
the circle $C_1$ of the Gauss diagram. 
Since the virtual link is odd, the number of endpoints of chords on $C_1$ is odd. 
Let $P_{1},P_{2},\ldots,P_{2n+1}$ be the endpoints on $C_1$ 
arranged in this order.

The exchange of the positions of $P_{1}$ and $P_{2}$ 
is realized by a combination of $\Xi$-moves 
as shown in Figure~\ref{pf-lem-odd}(1)--(4). 
More precisely, the arrangement (2) is obtained from (1) 
by applying the deformations in Lemma~\ref{lem-cross-move} $n-1$ times. 
Next we obtain (3) from (2) by a single $\Xi$-move 
exchanging the positions of $P_{1}$ and $P_{2n+1}$. 
Finally we slide $P_{1}$ and $P_{2}$ along $C_1$ 
to obtain (4) from (3). 
\end{proof}

\begin{figure}[htbp]
\centering
    \begin{overpic}[width=12cm]{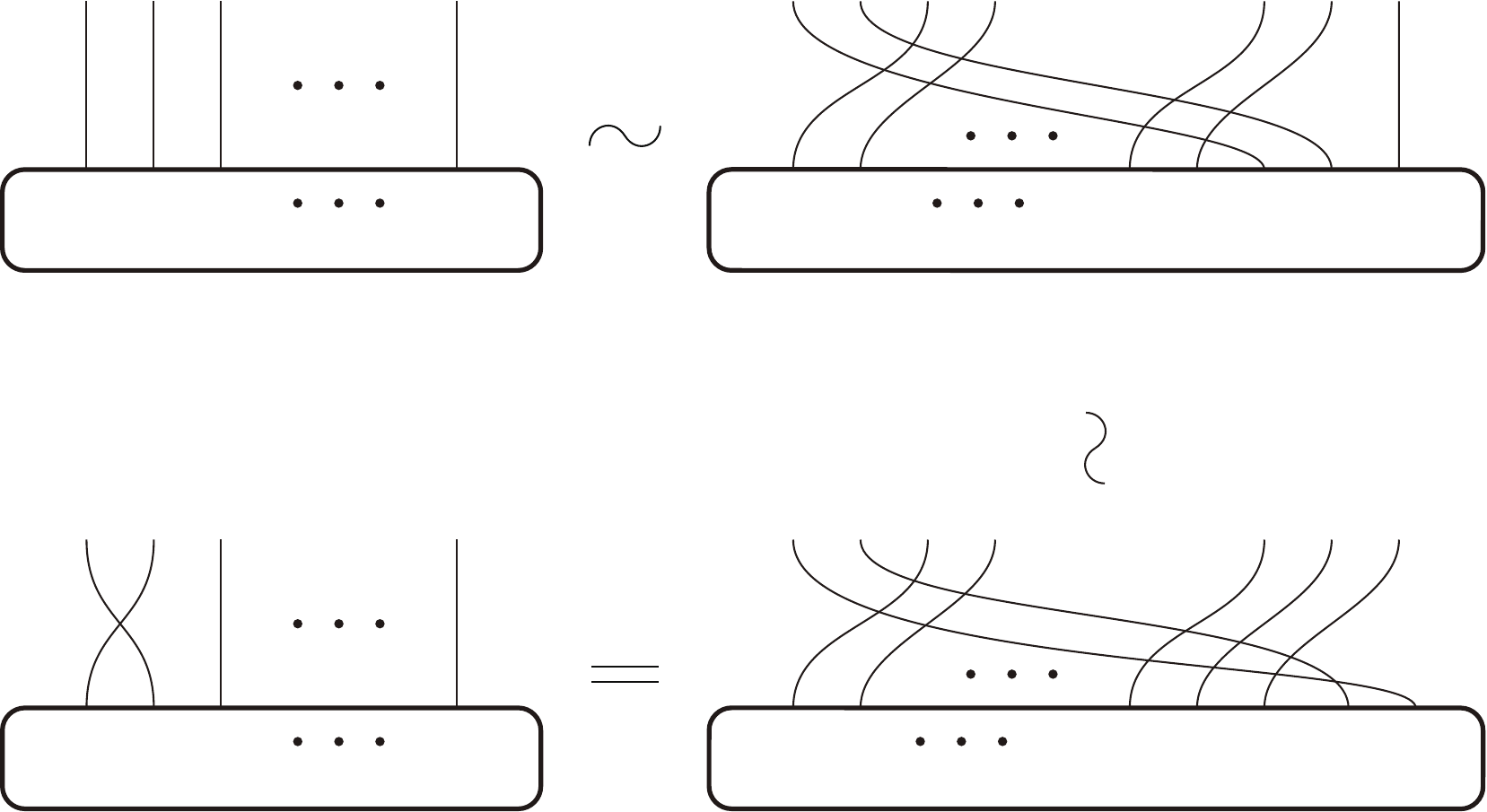}
      \put(54,109){(1)}
      \put(244,109){(2)}
      \put(54,-14){(4)}
      \put(244,-14){(3)}
      \put(-15,133){$C_{1}$}
      \put(126,164){Lem~\ref{lem-cross-move}}
      \put(260,80){$\Xi$}
      % (a)
      \put(15,137){$P_{1}$}
      \put(30.5,137){$P_{2}$}
      \put(46,137){$P_{3}$} 
      \put(95,137){$P_{2n+1}$}
      % (d)
      \put(15,13){$P_{2}$}
      \put(30,13){$P_{1}$}
      \put(46,13){$P_{3}$}
      \put(95,13){$P_{2n+1}$}
      % (b)
      \put(177,137){$P_{3}$}
      \put(194,137){$P_{4}$}
      \put(243,137){$P_{2n-1}$}
      \put(270.5,137){$P_{2n}$}
      \put(288,137){$P_{1}$}
      \put(301,137){$P_{2}$}
      \put(314,137){$P_{2n+1}$}
      % (c)
      \put(177,13){$P_{3}$}
      \put(194,13){$P_{4}$}
      \put(237,13){$P_{2n-1}$}
      \put(263,13){$P_{2n}$}
      \put(280.5,13){$P_{2n+1}$}
      \put(307,13){$P_{2}$}
      \put(321,13){$P_{1}$}
    \end{overpic}
  \vspace{1em}
  \caption{Proof of Lemma~\ref{lem-odd}}
  \label{pf-lem-odd}
\end{figure}

We define a map $e:\mathbb{Z}\rightarrow\{-1,0,1\}$ 
by $e(n)=1$ for $n>0$, 
$e(n)=-1$ for $n<0$, 
and $e(n)=0$ for $n=0$. 
For integers $a,b\in\mathbb{Z}$, 
we denote by $G(a,b)$ the Gauss diagram with two circles 
$C_{1}$ and $C_{2}$ as shown in the left of Figure~\ref{rep-odd}; that is, 
it consists of $|a|$ horizontal nonself-chords of type~$(1,2)$ with sign~$e(a)$ 
and $|b|$ horizontal nonself-chords of type~$(2,1)$ with sign~$e(b)$. 
In the right of the figure, we illustrate the Gauss diagram $G(3,-2)$. 
Let $L(a,b)$ denote the $2$-component virtual link 
presented by the Gauss diagram $G(a,b)$. 
Then we see that the set 
$$\{L(a,b)\mid a,b\in \mathbb{Z} \text{ with } a+b\equiv 1 \ ({\rm mod}~2)\}$$ 
is a complete representative system of $\Xi$-equivalence classes 
of $2$-component odd virtual links as follows.

\begin{figure}[htbp]
\centering
    \begin{overpic}[width=9cm]{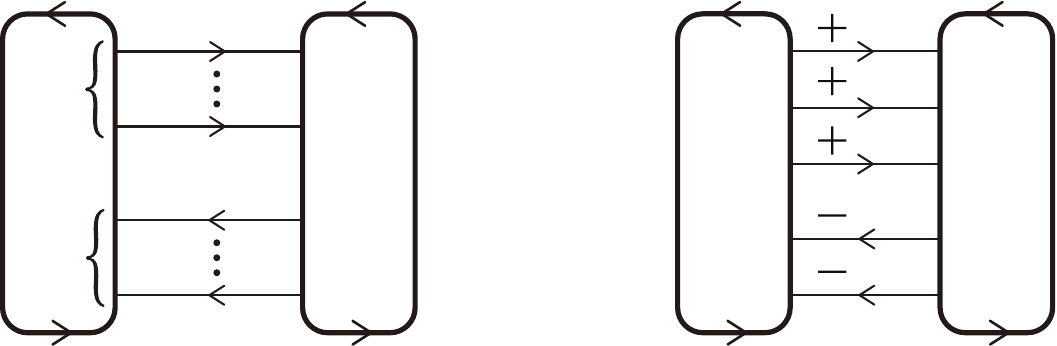}
      \put(8,60){$|a|$}
      \put(9,19){$|b|$}
      \put(-14,39){$C_{1}$}
      \put(105,39){$C_{2}$}
      \put(31,76){$e(a)$}
      \put(31,58){$e(a)$}
      \put(31,35){$e(b)$}
      \put(31,17){$e(b)$}
      %%%%%%%%%%%
      \put(149,39){$C_{1}$}
      \put(259,39){$C_{2}$}
      \put(37,-13){$G(a,b)$}
      \put(192,-13){$G(3,-2)$}
    \end{overpic}
  \vspace{1em}
  \caption{The Gauss diagram $G(a,b)$}
  \label{rep-odd}
\end{figure}

\begin{proposition}\label{prop-odd}
Let $L=K_{1}\cup K_{2}$ be a $2$-component odd virtual link. 
Then $L$ is $\Xi$-equivalent to $L(a,b)$, 
where $a={\rm Lk}(K_{1},K_{2})$ and $b={\rm Lk}(K_{2},K_{1})$. 
\end{proposition}

\begin{proof}
Let $G$ be a Gauss diagram of $L$ with two circles $C_{1}$ and $C_{2}$. 
By Lemma~\ref{lem-odd}, we can freely move the positions of endpoints on $C_{i}$ $(i=1,2)$ up to $\Xi$-equivalence. 
Therefore we deform every self-chord into a free chord, and remove it by an R1-move. 
Furthermore, we rearrange the nonself-chords horizontally so that the nonself-chords of type~$(1,2)$ are placed above those of type~$(2,1)$. 
If two consecutive nonself-chords of the same type have opposite signs, then we delete them by an R2-move. 
Finally $G$ is $\Xi$-equivalent to $G(a,b)$ for some $a,b\in\mathbb{Z}$. 
Therefore $L$ is $\Xi$-equivalent to the virtual link $L(a,b)$ presented by $G(a,b)$. 

By definition, $L(a,b)=K_{1}'\cup K_{2}'$ satisfies  ${\rm Lk}(K'_{1},K'_{2})=a$ and ${\rm Lk}(K'_{2},K'_{1})=b$. 
Since $L$ and $L(a,b)$ are $\Xi$-equivalent, we have ${\rm Lk}(K_{1},K_{2})=a$ and ${\rm Lk}(K_{2},K_{1})=b$ by Lemma~\ref{lem-Lk}. 
\end{proof}

\begin{proof}[Proof of {\rm Theorem~\ref{th-odd}}]
\underline{${\rm (i)}\Rightarrow{\rm (ii)}$.}~
This follows from Lemma~\ref{lem-Lk} directly. 

\underline{${\rm (ii)}\Rightarrow{\rm (i)}$.}~
By Proposition~\ref{prop-odd}, 
$L=K_{1}\cup K_{2}$ and $L'=K'_{1}\cup K'_{2}$ are $\Xi$-equivalent to $L(a,b)$ and $L(a',b')$, respectively, 
where 
\[
a={\rm Lk}(K_{1},K_{2}),\ b={\rm Lk}(K_{2},K_{1}),\ a'={\rm Lk}(K'_{1},K'_{2}),\ \text{and}\ b'={\rm Lk}(K'_{2},K'_{1}). 
\]
Since ${\rm Lk}(K_{i},K_{j})={\rm Lk}(K'_{i},K'_{j})$ holds for $\{i,j\}=\{1,2\}$,
we have $a=a'$ and $b=b'$. 
Therefore $L(a,b)$ is coincident to $L(a',b')$. 
\end{proof}

\begin{remark}
The deformation in Lemma~\ref{lem-odd} is called 
a {\it forbidden move} (cf.~\cite{GPV, Kan, Nel}). 
Lemma~\ref{lem-odd} implies that 
for a $2$-component odd virtual link, 
a forbidden move is realized by $\Xi$-moves. 
On the other hand, 
a $\Xi$-move is obviously realized by forbidden moves. 
Therefore the classification of $2$-component odd virtual links up to $\Xi$-moves 
is coincident with that up to forbidden moves, 
which is studied in \cite[Proposition~3.6]{ABMW} and \cite[Corollary~7]{Oka}.
\end{remark}

%%%%%%%%%% Shells %%%%%%%%%
\section{Shells}\label{sec-shell} 
To prove Theorem~\ref{th-even} for $2$-component even virtual links, we prepare several lemmas and propositions in Sections~\ref{sec-shell} and~\ref{sec-standardform}. 
It is not necessary to restrict the argument to $2$-component even virtual links. 
Therefore we do not assume that a $2$-component virtual link is even in these sections.

In a Gauss diagram, 
let $P_{1},P_{2}$, and $P_{3}$ be 
three consecutive endpoints  of chords 
lying on the same circle of the Gauss diagram. 
If $P_1$ and $P_3$ are connected by a self-chord, 
then the chord is called a {\it shell} 
(cf.~\cite{NNS}). 
See the left of Figure~\ref{shell}. 
Note that the orientation of a shell can be altered by a $\Xi$-move 
as shown in the right of the figure. 
In this sense, we may omit the orientation of a shell up to $\Xi$-equivalence in figures.

\begin{figure}[htbp]
\centering
    \begin{overpic}[width=8cm]{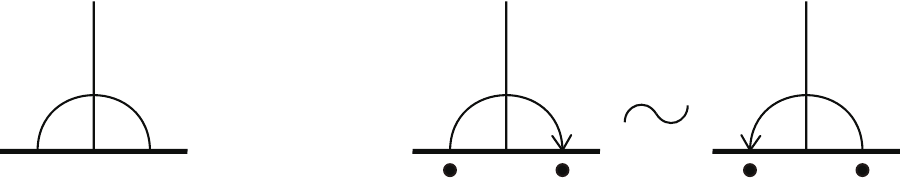}
      \put(4,-5){$P_{1}$}
      \put(19,-5){$P_{2}$}
      \put(34,-5){$P_{3}$}
      \put(112,19){$\e$}
      \put(188,19){$\e$}
      \put(163,22){$\Xi$}
    \end{overpic}
  \caption{A shell and its orientation}
  \label{shell}
\end{figure}

A {\em shell-pair} consists of a pair of shells whose four endpoints are consecutive 
on the same circle. 
If a shell-pair consists of one positive and one negative shells, 
then we can delete it by an R2-move 
(after choosing the orientations of the shells suitably). 
We say that a shell-pair is {\em positive} (resp. {\em negative}) if it consists of two positive (resp. two negative) shells. 
See Figure~\ref{shell-pair}.

\begin{figure}[htbp]
\centering
    \begin{overpic}[width=6cm]{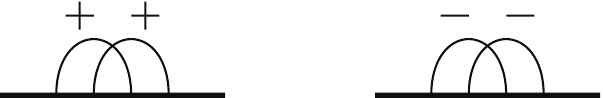}
    \end{overpic}
  \caption{A positive/negative shell-pair}
  \label{shell-pair}
\end{figure}

The following lemma allows us to move a shell-pair along 
a circle of a Gauss diagram freely.

\begin{lemma}\label{lem-sliding}
If two Gauss diagrams are related by a deformation as shown in {\rm Figure~\ref{sliding}}, then they are $\Xi$-equivalent.
\end{lemma}

\begin{figure}[htbp]
\centering
    \begin{overpic}[width=5cm]{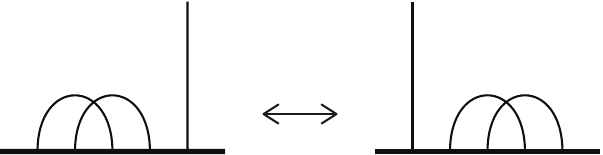}
      \put(6.5,17){$\e_{1}$}
      \put(28,17){$\e_{2}$}
      \put(104,17){$\e_{2}$}
      \put(126,17){$\e_{1}$}
    \end{overpic}
  \caption{Moving a shell-pair along a circle}
  \label{sliding}
\end{figure}

\begin{proof}
This follows by the deformation 
as shown in Figure~\ref{pf-lem-sliding}. 
\end{proof}

\begin{figure}[htbp]
\centering
    \begin{overpic}[width=9cm]{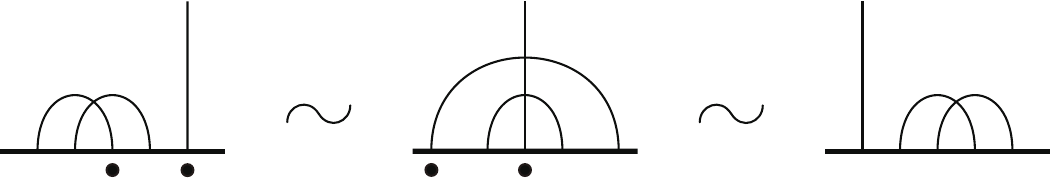}
      \put(7,23){$\e_{1}$}
      \put(29,23){$\e_{2}$}
      \put(100,23){$\e_{1}$}
      \put(110.5,16){$\e_{2}$}
      \put(216,23){$\e_{2}$}
      \put(239,23){$\e_{1}$}
      %%%
      \put(75,23){$\Xi$}
      \put(175,23){$\Xi$}
    \end{overpic}
  \caption{Proof of Lemma~\ref{lem-sliding}}
  \label{pf-lem-sliding}
\end{figure}

A pair of positive and negative shell-pairs 
can be canceled in the following sense. 

\begin{lemma}[{cf.~\cite[Fig. 13]{ST}}]\label{lem-canceling-pairs}
If two Gauss diagrams are related by a deformation as shown in {\rm Figure~\ref{canceling-pairs}}, then they are $\Xi$-equivalent. 
\end{lemma} 

\begin{figure}[htbp]
\centering
    \begin{overpic}[width=7cm]{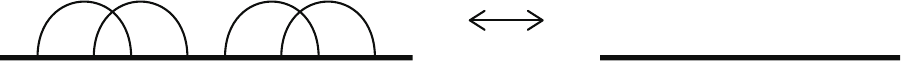}
      \put(11,15){$\e$}
      \put(34,15){$\e$}
      \put(46,15){$-\e$}
      \put(72,15){$-\e$}
    \end{overpic}
  \caption{Adding/canceling two consecutive shell-pairs with opposite signs}
  \label{canceling-pairs}
\end{figure}

\begin{proof}
This follows by the deformation 
as shown in Figure~\ref{pf-lem-canceling-pairs}. 
\end{proof}

\begin{figure}[htbp]
\centering
    \begin{overpic}[width=12.5cm]{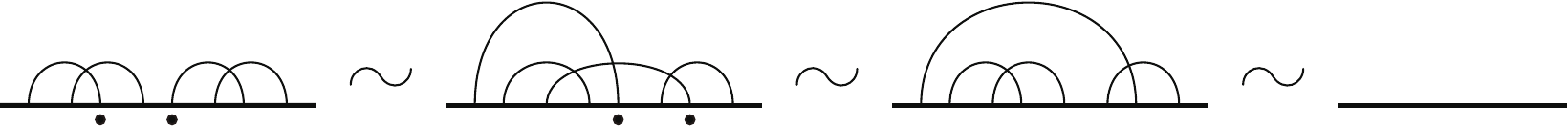}
      \put(83,18){$\Xi$}
      \put(184,18){$\Xi$}
      \put(282,18){R2}
      \put(279,-2){twice}
      %%%%%
      \put(9.5,17){$\e$}
      \put(26,17){$\e$}
      \put(36,17){$-\e$}
      \put(53,17){$-\e$}
      %%%%%
      \put(103.5,17){$\e$}
      \put(121,17){$\e$}
      \put(140,17){$-\e$}
      \put(156,17){$-\e$}
      %%%%%
      \put(205,17){$\e$}
      \put(219,17){$\e$}
      \put(228,17){$-\e$}
      \put(257,17){$-\e$}
    \end{overpic}
  \caption{Proof of Lemma~\ref{lem-canceling-pairs}}
  \label{pf-lem-canceling-pairs}
\end{figure}

We can exchange adjacent endpoints of chords 
(by ignoring shells) in the following sense.

\begin{lemma}\label{lem-exchange}
If two Gauss diagrams are related by a deformation $(1)$, $(2)$, or $(3)$ as shown in {\rm Figure~\ref{exchange}}, then they are $\Xi$-equivalent.
\end{lemma}

\begin{figure}[htbp]
\centering
    \begin{overpic}[width=12cm]{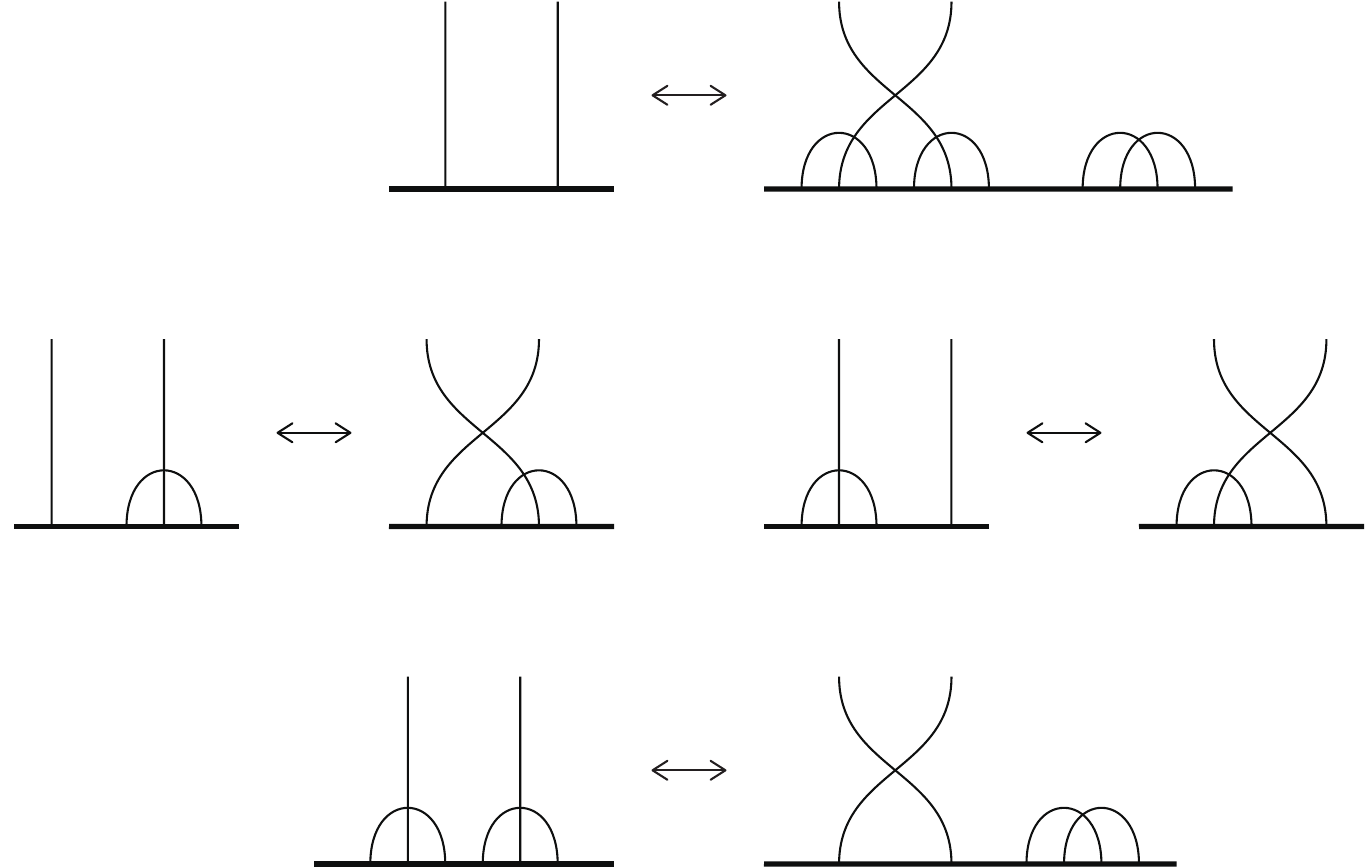}
      \put(166.5,201){$(1)$}
      \put(72.5,116){$(2)$}
      \put(260,116){$(2)$}
      \put(166.5,32){$(3)$}
      %%%% (1)
      \put(197,186){$\e_{1}$}
      \put(241,186){$\e_{2}$}
      \put(259,186){$-\e_{1}$}
      \put(289,186){$-\e_{2}$}
      %%%%% (2)
      \put(49,99){$\e$}
      \put(142,99){$\e$}
      \put(198,99){$\e$}
      \put(292,99){$\e$}
      %%%%% (3)
      \put(88,17){$\e_{1}$}
      \put(134,17){$\e_{2}$}
      \put(254,18){$\e_{1}$}
      \put(278,18){$\e_{2}$}
    \end{overpic}
  \caption{Deformations (1)--(3) in Lemma~\ref{lem-exchange}}
  \label{exchange}
\end{figure}

\begin{proof}
(3) This is realized by 
two $\Xi$-moves as shown in Figure~\ref{pf-lem-exchange}. 

\begin{figure}[H]
\centering
    \begin{overpic}[width=11cm]{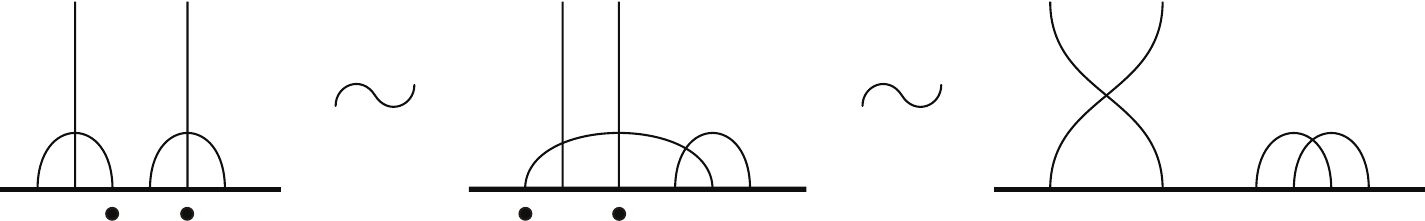}
      \put(79,35){$\Xi$}
      \put(194,35){$\Xi$}
      %%%%%
      \put(2,21){$\e_{1}$}
      \put(46,21){$\e_{2}$}
      %%%%%
      \put(110,18){$\e_{1}$}
      \put(153,23){$\e_{2}$}
      %%%%%
      \put(273,23){$\e_{1}$}
      \put(294,23){$\e_{2}$}
    \end{overpic}
  \caption{Proof of Lemma~\ref{lem-exchange}$(3)$}
  \label{pf-lem-exchange}
\end{figure}

(1) This is obtained by (3) and Lemma~\ref{lem-canceling-pairs} 
for $\e_1=\e_2$, and by
(3) and two R2-moves for $\e_1=-\e_2$. 

(2) This is realized by a single $\Xi$-move. 
\end{proof}

\begin{lemma}\label{lem-3shells}
If two Gauss diagrams are related by a deformation as shown in {\rm Figure~\ref{3shells}}, then they are $\Xi$-equivalent.
\end{lemma}

\begin{figure}[htbp]
\centering
    \vspace{1em}
    \begin{overpic}[width=6cm]{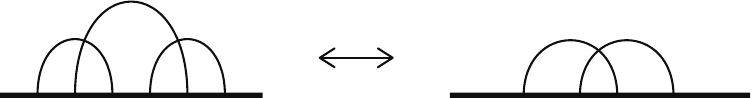}
      \put(5,16){$\e_{1}$}
      \put(26,26){$\e_{2}$}
      \put(45,16){$\e_{3}$}
      \put(119,17){$\e_{1}$}
      \put(144,17){$\e_{3}$}
    \end{overpic}
  \caption{A deformation in Lemma~\ref{lem-3shells}}
  \label{3shells}
\end{figure}

\begin{proof}
This follows by the deformation as shown in Figure~\ref{pf-lem-3shells}. 
\end{proof}

\begin{figure}[htbp]
\centering
    \vspace{1em}
    \begin{overpic}[width=10cm]{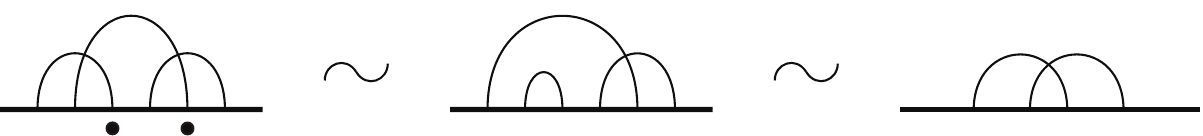}
      \put(5,22){$\e_{1}$}
      \put(27,33){$\e_{2}$}
      \put(48,22){$\e_{3}$}
      \put(82,23){$\Xi$}
      \put(124.5,18.5){$\e_{2}$}
      \put(129,33){$\e_{1}$}
      \put(155,22){$\e_{3}$}
      \put(185.5,23){R1}
      \put(229,22){$\e_{1}$}
      \put(257,22){$\e_{3}$}
    \end{overpic}
  \caption{Proof of Lemma~\ref{lem-3shells}}
  \label{pf-lem-3shells}
\end{figure}

\begin{proposition}\label{prop-selfchord}
Any Gauss diagram of a $2$-component virtual link is $\Xi$-equivalent to a Gauss diagram with two circles $C_{1}$ and $C_{2}$ which satisfies the following conditions. 
\begin{enumerate}
\item[(i)] Any self-chord on $C_i$ is a shell $(i=1,2)$. 
\item[(ii)] All nonself-chords are arranged horizontally such that the nonself-chords of type~$(1,2)$ are placed above those of type~$(2,1)$. 
\end{enumerate}
\end{proposition}

{\rm Figure~\ref{ex-prop}} shows an example of a Gauss diagram satisfying the conditions 
(i) and (ii) in Proposition~\ref{prop-selfchord}. 
We remark that any self-chord is either located around an endpoint of some nonself-chord 
or forms a shell-pair by the conditions.

\begin{figure}[htbp]
\centering
    \begin{overpic}[width=4cm]{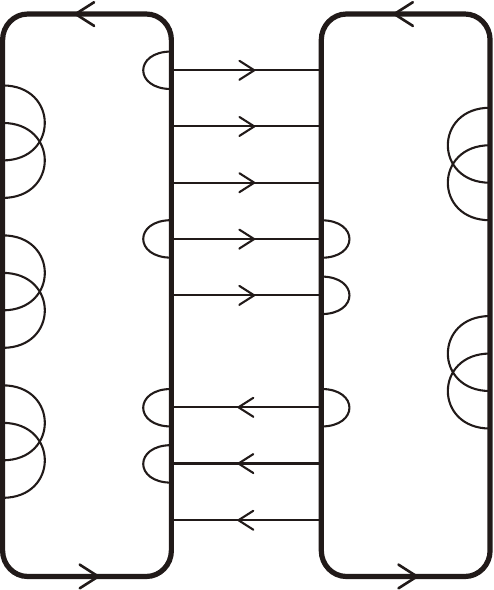}
      \put(-18,65){$C_{1}$}
      \put(119,65){$C_{2}$}
    \end{overpic}
  \caption{A Gauss diagram satisfying Proposition~\ref{prop-selfchord}} 
  \label{ex-prop} 
\end{figure}

\begin{proof}[Proof of {\rm Proposition~\ref{prop-selfchord}}]
In a given Gauss diagram, let $\gamma_1,\dots,\gamma_n$ be 
\emph{all} self-chords on $C_1$, and let $P_i$ and $Q_i$ be the endpoints of $\gamma_i$ 
$(i=1,\dots,n)$. 

We first slide $P_1$ to $Q_1$ along the circle $C_1$ 
by using the deformations (1) and (2) in Lemma~\ref{lem-exchange} 
as follows (see also Figure~\ref{pf-prop-selfchord1}).
Let $X_1,\dots,X_s$ 
be the endpoints of self-/nonself-chords between $P_1$ and $Q_1$ lying in this order. 
First we exchange the positions of $P_1$ and $X_1$ 
by using the deformation (1) which adds a shell at $P_1$, 
a shell at $X_1$, and a shell-pair on $C_1$. 
We may ignore the position of the shell-pair up to $\Xi$-equivalence 
by Lemma~\ref{lem-sliding}, and omit it in Figure~\ref{pf-prop-selfchord1}.  
Next we exchange the positions of $P_1$ and $X_2$ 
by the deformation (2), which removes the shell at $P_1$ 
and adds a shell at $X_2$. 
By applying the deformations (1) and (2) alternately 
to exchange $P_1$ and $X_i$ $(i=3,\dots,s)$, 
$\gamma_1$ is finally deformed into a free chord 
if $s$ is even, or it forms a shell-pair if $s$ is odd. 
For $s$ even, we remove $\gamma_{1}$ by an R1-move. 
may be deformed into a self-chord that is not a shell.

\begin{figure}[htbp]
\centering
    \vspace{1em}
    \begin{overpic}[width=12.5cm]{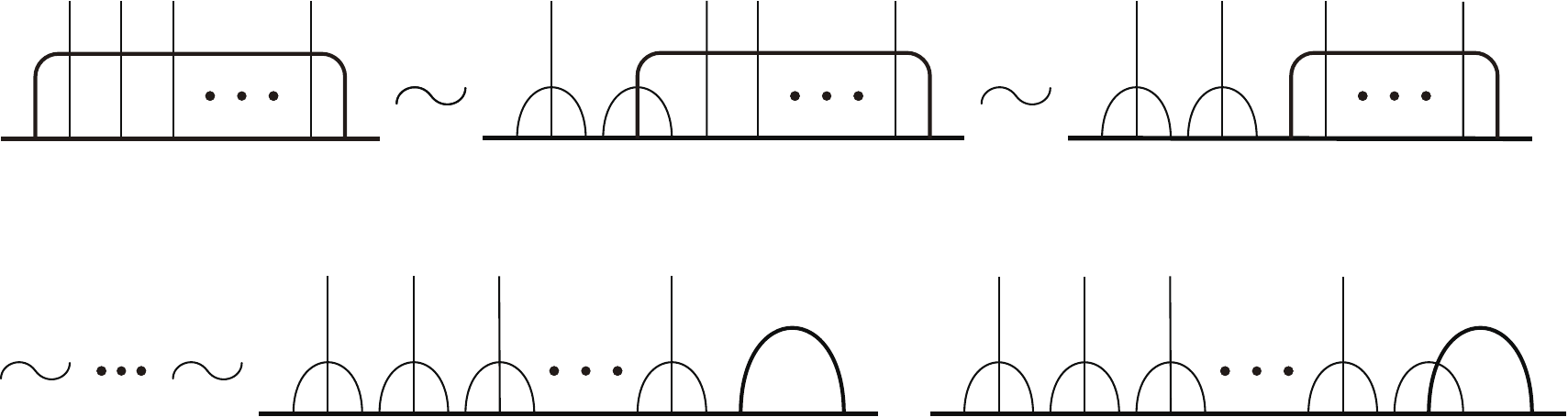}
      %%% upper
      \put(92,81){(1)}
      \put(224,81){(2)}
      %%%%%
      \put(2,88){$\gamma_{1}$}
      \put(0,51){$P_{1}$}
      \put(11,51){$X_{1}$}
      \put(24,51){$X_{2}$}
      \put(37,51){$X_{3}$}
      \put(51,54){$\dots$}
      \put(63.5,51){$X_{s}$}
      \put(76,51){$Q_{1}$}
      %%%%%
      \put(142,87){$\gamma_{1}$}
      \put(120,51){$X_{1}$}
      \put(140,51){$P_{1}$}
      \put(155,51){$X_{2}$}
      \put(169,51){$X_{3}$}
      \put(184,54){$\dots$}
      \put(197,51){$X_{s}$}
      \put(210,51){$Q_{1}$}
      %%%%%
      \put(290,87){$\gamma_{1}$}
      \put(253,51){$X_{1}$}
      \put(272,51){$X_{2}$}
      \put(287,51){$P_{1}$}
      \put(298,51){$X_{3}$}
      \put(313,54){$\dots$}
      \put(326,51){$X_{s}$}
      \put(340,51){$Q_{1}$}
      
      %%% lower
      \put(2,19){(1)}
      %%%
      \put(175,25){$\gamma_{1}$}
      \put(110,-28){($s$ is even)}
      \put(69,-12){$X_{1}$}
      \put(89,-12){$X_{2}$}
      \put(108,-12){$X_{3}$}
      \put(129,-9){$\dots$}
      \put(147,-12){$X_{s}$}
      \put(164,-12){$P_{1}$}
      \put(186,-12){$Q_{1}$}
      %%%%%
      \put(201,7){or}
      \put(332,25){$\gamma_{1}$}
      \put(264,-28){($s$ is odd)}
      \put(221,-12){$X_{1}$}
      \put(241,-12){$X_{2}$}
      \put(261,-12){$X_{3}$}
      \put(281,-9){$\dots$}
      \put(300,-12){$X_{s}$}
      \put(319.5,-12){$P_{1}$}
      \put(343,-12){$Q_{1}$}
    \end{overpic}
  \vspace{2em}
  \caption{Sliding $P_{1}$ to $Q_{1}$ along $C_{1}$}
  \label{pf-prop-selfchord1}
\end{figure}

In the obtained Gauss diagram, 
let $Y_1,\dots,Y_t$ be the endpoints of chords 
between $P_2$ and $Q_2$ except for the endpoints of shells 
produced in the above sliding process of $P_1$ to $Q_1$. 
Similarly to the sliding of $P_1$ to $Q_1$, 
we slide $P_2$ to $Q_2$ along $C_1$ 
by exchanging $P_2$ and $Y_1,\dots,Y_t$ step by step. 
Since $P_{2}$, $Q_{2}$, and $Y_{i}$ may have shells, 
we shall use the deformation (3) in Lemma~\ref{lem-exchange} in addition to 
the deformations (1) and (2) as follows. 
If both $P_2$ and $Y_i$ have shells for some $i$,  
we exchange their positions by using the deformation (3). 
This removes the shells at $P_2$ and at $Y_i$, 
and adds a shell-pair on $C_1$. 
Likewise, if $P_2$ has no shell but $Y_i$ does 
(resp. both $P_{2}$ and $Y_{i}$ have no shells) 
for some $i$,  
we exchange their positions by using the deformation~(2) 
(resp. deformation~(1)). 
Applying this sliding process of $P_{2}$ to $Q_{2}$ leads to four cases 
depending on 
the number of shells at $P_2$ and $Q_2$ 
as shown in Figure~\ref{pf-prop-selfchord2}. 
In the first case of the figure, $\gamma_{2}$ is a free chord and we remove it by 
an R1-move. 
In the second and third cases, we have a shell-pair. 
In the last case, 
we use Lemma~\ref{lem-3shells} to obtain 
a shell-pair. 
For example, we consider the case $t=5$ 
such that $P_{2}$, $Q_{2}$, $Y_{1}$, $Y_{2}$, and $Y_{3}$ have shells 
as shown in the top-left of Figure~\ref{pf-prop-selfchord3}. 
By the sliding process of $P_{2}$ to $Q_{2}$, 
the self-chord $\gamma_{2}$ is deformed into the third cases of Figure~\ref{pf-prop-selfchord2} 
by applying the deformations (3), (2), (3), (1), and (2) in this order. 
Furthermore, $Y_{1}$, $Y_{2}$, and $Y_{3}$ lose the shells, and $Y_{4}$ and $Y_{5}$ get shells.

\begin{figure}[htbp]
\centering
    \vspace{1em}
    \begin{overpic}[width=10cm]{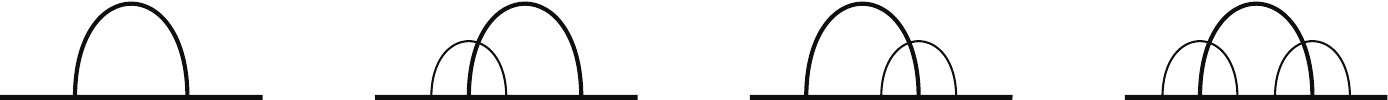}
      \put(22,26){$\gamma_{2}$}
      \put(9,-12){$P_{2}$}
      \put(35,-12){$Q_{2}$}
      %%%%%
      \put(104,26){$\gamma_{2}$}
      \put(91,-12){$P_{2}$}
      \put(117,-12){$Q_{2}$}
      %%%%%
      \put(172,26){$\gamma_{2}$}
      \put(159,-12){$P_{2}$}
      \put(185,-12){$Q_{2}$}
      %%%%%
      \put(253,26){$\gamma_{2}$}
      \put(240,-12){$P_{2}$}
      \put(266,-12){$Q_{2}$}
    \end{overpic}
  \vspace{1em}
  \caption{Four possible cases}
  \label{pf-prop-selfchord2}
\end{figure}

\begin{figure}[htbp]
\centering
    \begin{overpic}[width=12cm]{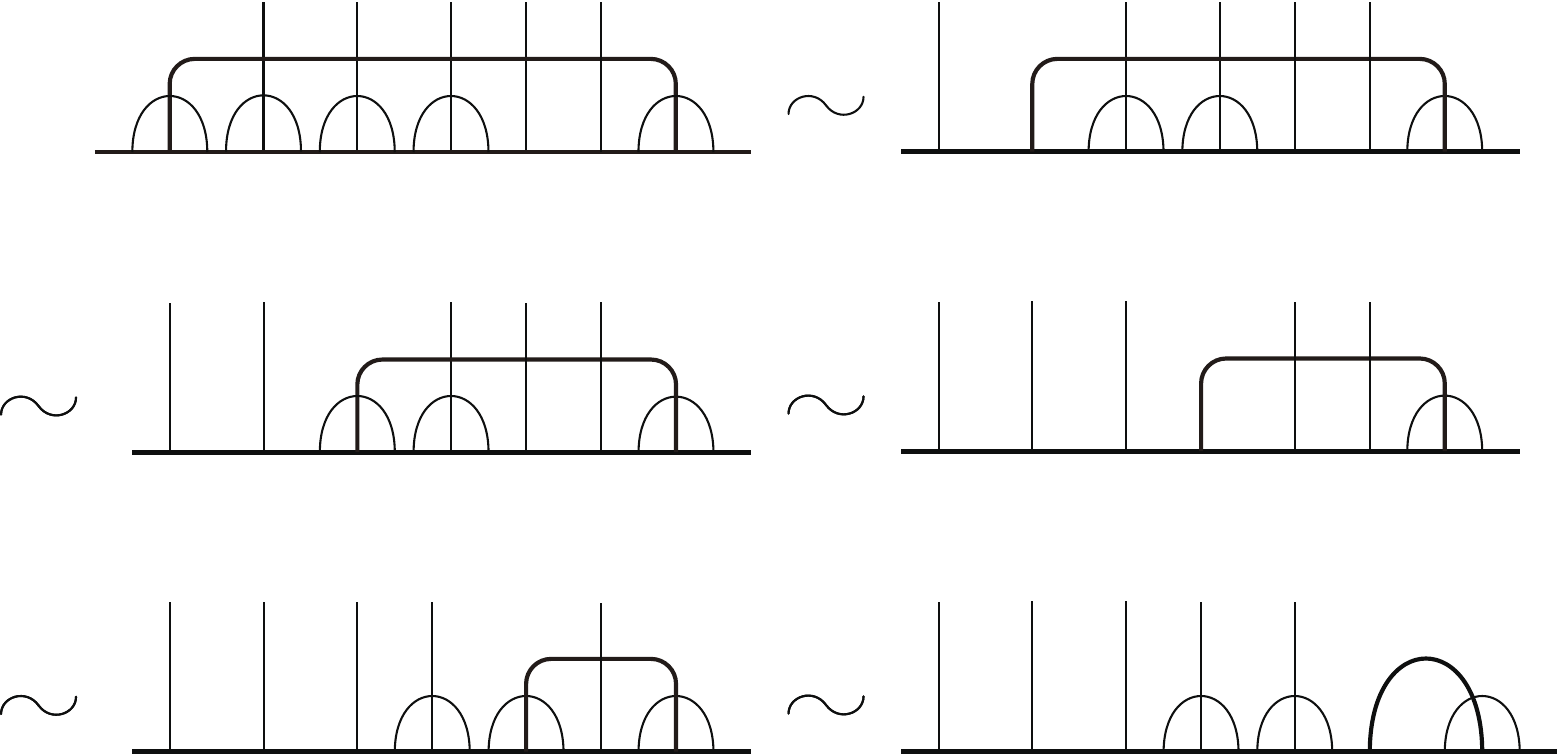}
      \put(175,151){(3)}
      \put(2,85){(2)}
      \put(175,85){(3)}
      \put(2,19){(1)}
      \put(175,19){(2)}
%%%%%%%%%% top
      \put(34,158){$\gamma_{2}$}
      \put(32,118){$P_{2}$}
      \put(54,118){$Y_{1}$}
      \put(75,118){$Y_{2}$}
      \put(95,118){$Y_{3}$}
      \put(112,118){$Y_{4}$}
      \put(128,118){$Y_{5}$}
      \put(144,118){$Q_{2}$}
      %%%
      \put(223,158){$\gamma_{2}$}
      \put(202,118){$Y_{1}$}
      \put(221,118){$P_{2}$}
      \put(242,118){$Y_{2}$}
      \put(262,118){$Y_{3}$}
      \put(280,118){$Y_{4}$}
      \put(296,118){$Y_{5}$}
      \put(313,118){$Q_{2}$}
%%%%%%%%%%% middle
      \put(75,92){$\gamma_{2}$}
      \put(33,53){$Y_{1}$}
      \put(54,53){$Y_{2}$}
      \put(73,53){$P_{2}$}
      \put(95,53){$Y_{3}$}
      \put(112,53){$Y_{4}$}
      \put(128,53){$Y_{5}$}
      \put(144,53){$Q_{2}$}
      %%%
      \put(262,92){$\gamma_{2}$}
      \put(202,53){$Y_{1}$}
      \put(222,53){$Y_{2}$}
      \put(242,53){$Y_{3}$}
      \put(260,53){$P_{2}$}
      \put(280,53){$Y_{4}$}
      \put(296,53){$Y_{5}$}
      \put(313,53){$Q_{2}$}
%%%%%%%%%%% bottom
      \put(112,26){$\gamma_{2}$}
      \put(33,-12){$Y_{1}$}
      \put(54,-12){$Y_{2}$}
      \put(74,-12){$Y_{3}$}
      \put(91,-12){$Y_{4}$}
      \put(110,-12){$P_{2}$}
      \put(128,-12){$Y_{5}$}
      \put(144,-12){$Q_{2}$}
      %%%
      \put(307,27){$\gamma_{2}$}
      \put(202,-12){$Y_{1}$}
      \put(222,-12){$Y_{2}$}
      \put(242,-12){$Y_{3}$}
      \put(260,-12){$Y_{4}$}
      \put(280,-12){$Y_{5}$}
      \put(296,-12){$P_{2}$}
      \put(320,-12){$Q_{2}$}
    \end{overpic}
  \vspace{1em}
  \caption{An example of the sliding process of $P_{2}$ to $Q_{2}$}
  \label{pf-prop-selfchord3}
\end{figure}

We slide $P_i$ to $Q_i$ for $i=3,\dots,n$ similarly to the case $i=2$ 
to obtain a Gauss diagram where 
any self-chord on $C_1$ is a shell. 
We perform a similar deformation for 
the self-chords on $C_2$ 
to obtain a Gauss diagram satisfying the condition (i).

Similarly to the deformations of self-chords as above, 
we deform the nonself-chords by Lemma~\ref{lem-exchange} 
so that the obtained Gauss diagram satisfies the condition (ii) in addition to (i). 
\end{proof}

%%%%%%%%%% Standard form %%%%%%%%%
\section{Ladders}\label{sec-standardform} 

In this section, we introduce the notion of ladders 
and give a representative of the $\Xi$-equivalence class of a 
$2$-component virtual link (Proposition~\ref{prop-even}). 

The sign of a shell can be altered with making a shell-pair as follows. 

\begin{lemma}\label{lem-sign}
If two Gauss diagrams are related by a deformation as shown in {\rm Figure~\ref{sign-shell}}, then they are $\Xi$-equivalent. 
\end{lemma} 

\begin{figure}[htbp]
\centering
    \begin{overpic}[width=6cm]{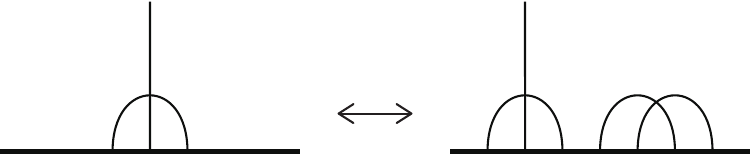}
      \put(22,13){$\e$}
      \put(100,13){$-\e$}
      \put(134,13){$\e$}
      \put(160,13){$\e$}
    \end{overpic}
  \caption{Altering the sign of a shell with a shell-pair}
  \label{sign-shell}
\end{figure}

\begin{proof}
This follows by the deformation 
as shown in Figure~\ref{pf-lem-sign}. 
\end{proof}

\begin{figure}[htbp]
\centering
    \begin{overpic}[width=12cm]{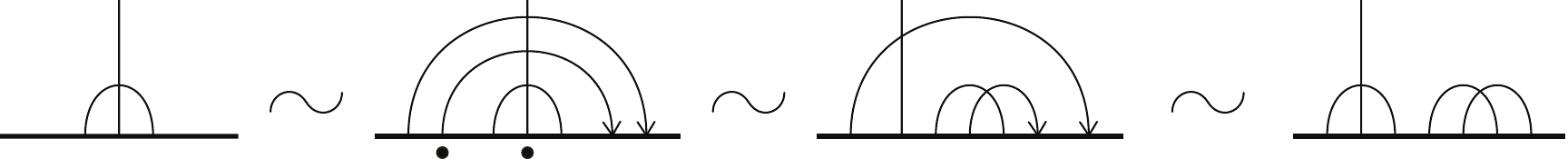}
      \put(60.5,20){R2}
      \put(160,20){$\Xi$}
      \put(246,20){Lem~\ref{lem-sliding}}
      %%%
      \put(15,15){$\e$}
      %%%
      \put(89,30){$-\e$}
      \put(97,19.3){$\e$}
      \put(105,15){$\e$}
      %%%
      \put(177,24){$-\e$}
      \put(201,15){$\e$}
      \put(225,15){$\e$}
      %%%
      \put(276,10){$-\e$}
      \put(308,15){$\e$}
      \put(333,15){$\e$}
    \end{overpic}
  \caption{Proof of Lemma~\ref{lem-sign}}
  \label{pf-lem-sign}
\end{figure}

By Lemmas~\ref{lem-sliding} and~\ref{lem-sign}, we may ignore the signs of shells and the positions of shell-pairs up to $\Xi$-equivalence, and we will often omit them in figures in the rest of this section. 

We now consider a portion of a Gauss diagram 
consisting of two parallel arcs  on $C_1$ and $C_2$ such that 
$C_1$ is oriented upwards and $C_2$ is oriented downwards together with
horizontal nonself-chords of type $(1,2)$ possibly with shells.
Such a portion of the Gauss diagram is called a {\it $(1,2)$-ladder}. 
Figure~\ref{ex-product} shows 
an example of a $(1,2)$-ladder 
with five nonself-chords.

\begin{figure}[htbp]
\centering
    \vspace{1em}
    \begin{overpic}[width=1.75cm]{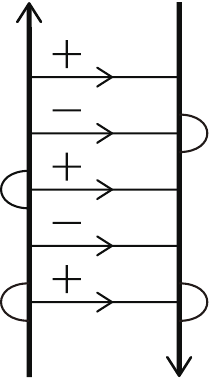}
      \put(1.5,95){$C_{1}$}
      \put(38,95){$C_{2}$}
    \end{overpic}
  
  \caption{The $(1,2)$-ladder $A_{1}A_{2}^{-1}A_{3}A_{1}^{-1}A_{4}$}
  \label{ex-product}
\end{figure}

In a $(1,2)$-ladder, 
the nonself-chords of type~$(1,2)$ 
are divided into eight classes 
labeled $A_{1}^{\e},A_{2}^{\e},A_{3}^{\e}$, and $A_{4}^{\e}$ $(\e=\pm1)$ 
as shown in Figure~\ref{labels}, 
where $\varepsilon$ is the sign of a nonself-chord. 
More precisely, a chord labeled $A_{1}^{+1}$ or $A_{1}^{-1}$ 
has no shells. 
On the other hand, 
a chord labeled $A_{2}^{+1}$ or $A_{2}^{-1}$ 
(resp. $A_{3}^{+1}$ or $A_{3}^{-1}$) has a shell 
at the endpoint on $C_2$ (resp. $C_1$), 
and a chord labeled $A_{4}^{+1}$ or $A_{4}^{-1}$ 
has a pair of shells at both endpoints. 
Then the word on the letters $\{A_{1},A_{2},A_{3},A_{4}\}$ of a $(1,2)$-ladder is obtained by reading the labels of nonself-chords in the ladder from top to bottom. 
For example, the $(1,2)$-ladder in Figure~\ref{ex-product} 
is expressed by $A_{1}A_{2}^{-1}A_{3}A_{1}^{-1}A_{4}$, where 
we abbreviate $X^{+1}$ to $X$ simply 
for $X=A_{1},A_{2},A_{3},A_{4}$. 
In what follows, we identify a $(1,2)$-ladder and its word on $\{A_{1},A_{2},A_{3},A_{4}\}$.

\begin{figure}[htbp]
\centering
    \vspace{1em}
    \begin{overpic}[width=11cm]{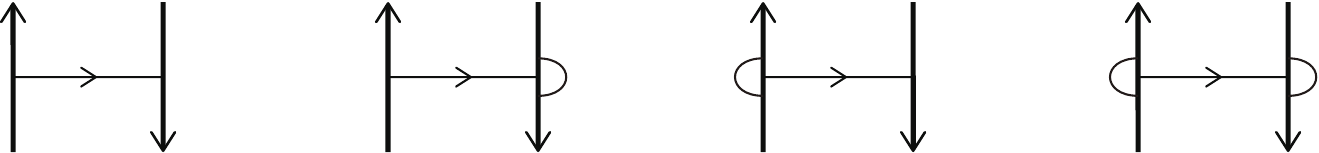}
      \put(-2,41){$C_{1}$}
      \put(34,41){$C_{2}$}
      \put(19,23){$\e$}
      \put(16,-14){$A_{1}^{\e}$}
      %%%
      \put(87,41){$C_{1}$}
      \put(123,41){$C_{2}$}
      \put(108,23){$\e$}
      \put(105,-14){$A_{2}^{\e}$}
      %%%
      \put(176,41){$C_{1}$}
      \put(212,41){$C_{2}$}
      \put(197,23){$\e$}
      \put(194,-14){$A_{3}^{\e}$}
      %%%
      \put(265,41){$C_{1}$}
      \put(301,41){$C_{2}$}
      \put(286,23){$\e$}
      \put(283,-14){$A_{4}^{\e}$}
    \end{overpic}
  \vspace{1em}
  \caption{The labels $A_{1}^{\e},A_{2}^{\e},A_{3}^{\e}$ and $A_{4}^{\e}$ of nonself-chords of type~$(1,2)$}
  \label{labels}
\end{figure}

We study the $\Xi$-equivalence classes 
of $(1,2)$-ladders in Lemmas~\ref{lem-inverse}--\ref{lem-nonself-chord}. 
Let $\emptyset$ denote the empty word 
expressing the ladder with no chord. 
We remark that the product of $\emptyset$ and 
any word $W$ on $\{A_{1},A_{2},A_{3},A_{4}\}$
is equal to $W$ itself by definition. 

\begin{lemma}\label{lem-inverse}
For $i\in \{1,2,3,4\}$,  we have the $\Xi$-equivalence 
$$A_iA_i^{-1}\sim A_i^{-1}A_i\sim\emptyset$$ 
up to shell-pairs. 
\end{lemma}

\begin{proof}
If $i=1$, then $A_{1}A_{1}^{-1}$ and $A_{1}^{-1}A_{1}$ are related to $\emptyset$ by R2-moves. 
Figure~\ref{pf-lem-inverse} shows the proofs for $i=2, 4$. 
The proof for $i=3$ is similar to that for $i=2$. 
\end{proof}

\begin{figure}[htbp]
\centering
    \begin{overpic}[width=12.5cm]{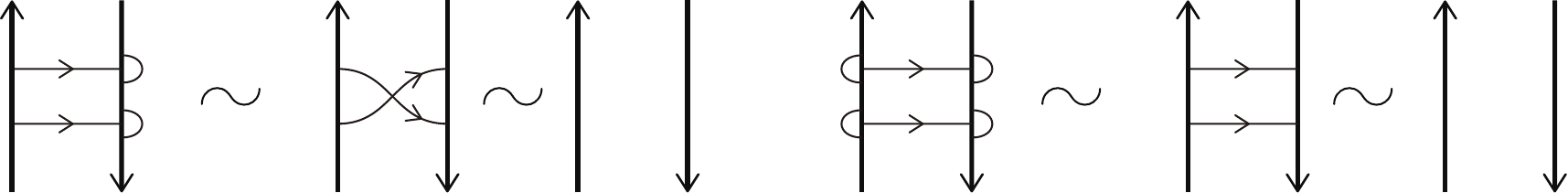}
      \put(1,-15){$A_{2}^{\e}A_{2}^{-\e}$}
      \put(142,-15){$\emptyset$}
      \put(36.5,28){Lem~\ref{lem-exchange}}
      \put(110,28){R2}
      \put(13,33){$\e$}
      \put(9,6){$-\e$}
      \put(81,31){$\e$}
      \put(78,8){$-\e$}
      %%%
      \put(195,-15){$A_{4}^{\e}A_{4}^{-\e}$}
      \put(338,-15){$\emptyset$}
      \put(229.5,28){Lem~\ref{lem-exchange}}
      \put(303,28){R2}
      \put(206,33){$\e$}
      \put(202,6){$-\e$}
      \put(276,33){$-\e$}
      \put(280,6){$\e$}
    \end{overpic}
  \vspace{1em}
  \caption{Proofs of Lemma~\ref{lem-inverse} for $i=2,4$}
  \label{pf-lem-inverse}
\end{figure}

By Lemma~\ref{lem-inverse}, 
if two words $W$ and $V$ on $\{A_{1},A_{2},A_{3},A_{4}\}$ 
satisfy $WV\sim\emptyset$, 
then we have $$W\sim W(VV^{-1})
\sim(WV)V^{-1}\sim V^{-1}$$ and $V\sim W^{-1}$, 
where $W^{-1}$ and $V^{-1}$ are the inverse words of $W$ and $V$, respectively. 
We remark that the $\Xi$-equivalence class of any $(1,2)$-ladder up to shell-pairs 
by an element of the free group generated by $\{A_{1},A_{2},A_{3},A_{4}\}$. 

\begin{lemma}\label{lem-commutability}
We have the following $\Xi$-equivalence up to shell-pairs. 
\begin{enumerate}
\item[(i)] $A_{1}A_{2}\sim A_{2}A_{1}\sim A_{3}A_{4}\sim A_{4}A_{3}$.  
\item[(ii)] $A_{1}A_{3}\sim A_{3}A_{1}\sim A_{2}A_{4}\sim A_{4}A_{2}$. 
\item[(iii)] $A_{1}A_{4}\sim A_{4}A_{1}$ and $A_{2}A_{3}\sim A_{3}A_{2}$. 
\end{enumerate}
In particular, the letters $A_{1}$, $A_{2}$, $A_{3}$, and $A_{4}$ 
are mutually commutative. 
\end{lemma}

\begin{proof}
(i) Figure~\ref{pf-lem-commutability} shows the proof. 
We remark that the R3-move in the figure 
is the same as the one used at the top  of Figure~\ref{R3-Gauss}.

\begin{figure}[htbp]
\centering
    \begin{overpic}[width=12cm]{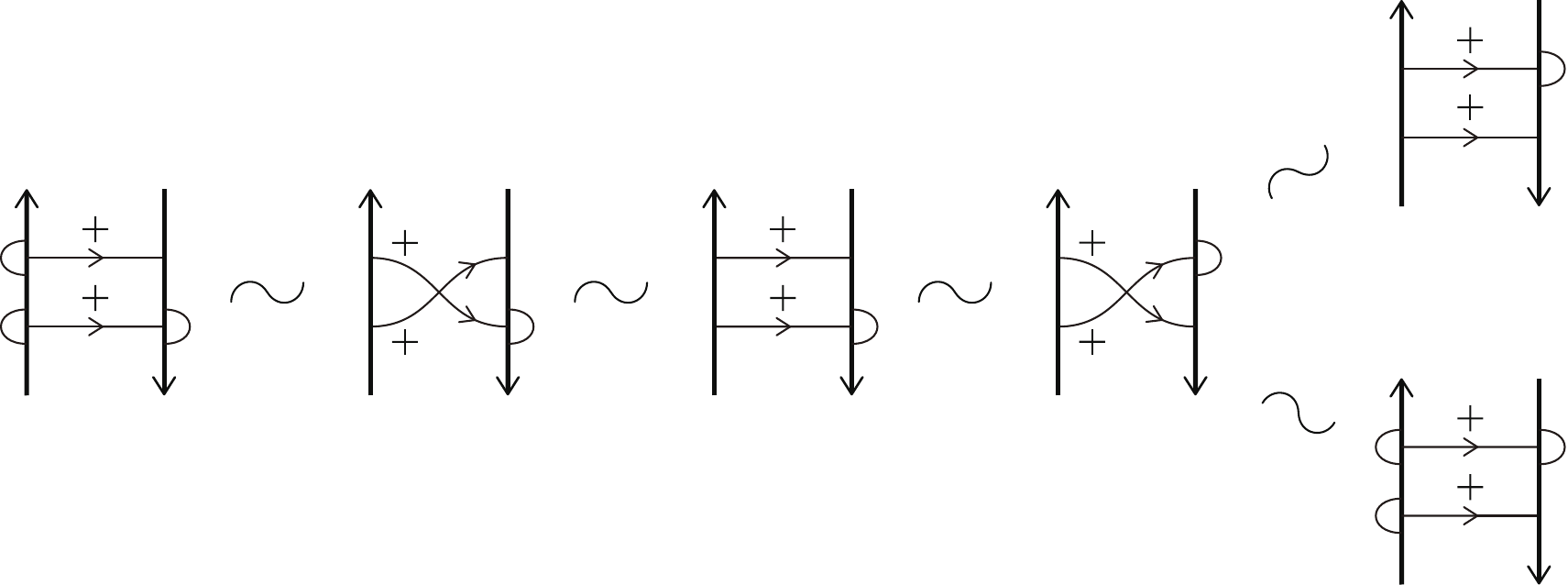}
      \put(40.5,70){Lem~\ref{lem-exchange}}
      \put(116,70){Lem~\ref{lem-exchange}}
      \put(202,70){R3}
      \put(255,100){Lem~\ref{lem-exchange}}
      \put(255,20){Lem~\ref{lem-exchange}}
      %%%
      \put(10,27){$A_{3}A_{4}$}
      \put(160,27){$A_{1}A_{2}$}
      \put(310,69){$A_{2}A_{1}$}
      \put(310,-13){$A_{4}A_{3}$}
    \end{overpic}
  \caption{Proof of Lemma~\ref{lem-commutability}(i)}
  \label{pf-lem-commutability}
\end{figure}

(ii) The proof is similar to that of (i) 
by exchanging $A_{2}$ and $A_{3}$.

(iii) It holds that $A_{4}\sim A_{3}^{-1}A_{2}A_{1}$ by $A_{3}A_{4}\sim A_{2}A_{1}$ in (i) and Lemma~\ref{lem-inverse}, 
$A_{1}A_{3}^{-1}\sim A_{3}^{-1}A_{1}$ by $A_{1}A_{3}\sim A_{3}A_{1}$ in (ii), and 
$A_{1}A_{2}\sim A_{2}A_{1}$ in (i). 
Therefore we have 
\[
\begin{split}
&A_1A_4\sim A_1(A_3^{-1}A_2A_1)= (A_1A_3^{-1})A_2A_1\sim (A_3^{-1}A_1)A_2A_1 \\
& = A_3^{-1}(A_1A_2)A_1
\sim A_3^{-1}(A_2A_1)A_1=(A_3^{-1}A_2A_1)A_1 \sim A_4A_1. 
\end{split}
\]
Similarly, it holds that $A_{3}\sim A_{4}^{-1}A_{1}A_{2}$ by $A_{4}A_{3}\sim A_{1}A_{2}$ in (i), 
$A_{2}A_{4}^{-1}\sim A_{4}^{-1}A_{2}$ by $A_{2}A_{4}\sim A_{4}A_{2}$ in (ii), and 
$A_{1}A_{2}\sim A_{2}A_{1}$ in (i). 
Therefore we have 
\[
A_{2}A_{3}\sim A_{2}A_{4}^{-1}A_{1}A_{2}\sim A_{4}^{-1}A_{2}A_{1}A_{2}\sim A_{4}^{-1}A_{1}A_{2}A_{2}\sim A_{3}A_{2}. 
\] 
\end{proof}

\begin{lemma}\label{lem-square}
We have the $\Xi$-equivalence $A_{2}^{2}\sim A_{3}^{2}$ up to shell-pairs. 
\end{lemma}

\begin{proof}
This follows by the deformation 
as shown in Figure~\ref{pf-lem-square}.
\end{proof}

\begin{figure}[htbp]
\centering
    \begin{overpic}[width=7cm]{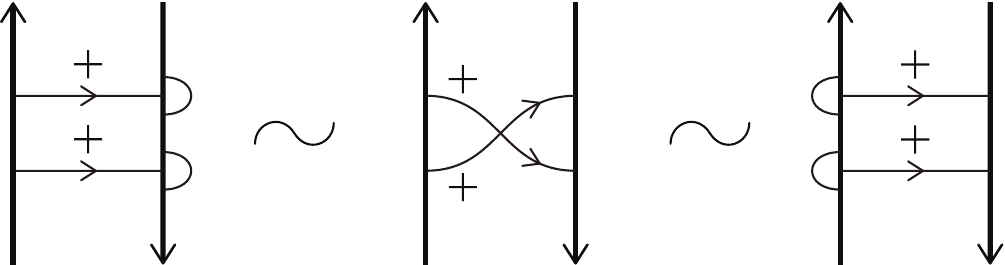}
      \put(12,-14){$A_{2}^{2}$}
      \put(177,-14){$A_{3}^{2}$}
      \put(43,32){Lem~\ref{lem-exchange}}
      \put(121,32){Lem~\ref{lem-exchange}}
    \end{overpic}
  \vspace{1em}
  \caption{Proof of Lemma~\ref{lem-square}}
  \label{pf-lem-square}
\end{figure}

\begin{lemma}\label{lem-nonself-chord}
Any $(1,2)$-ladder is $\Xi$-equivalent to either $A_{1}^{a_{1}}A_{2}^{a_{2}}$ or $A_{1}^{a_{1}}A_{2}^{a_{2}}A_{3}$ for some $a_{1},a_{2}\in\mathbb{Z}$ up to  shell-pairs. 
\end{lemma}

\begin{proof}
By Lemmas~\ref{lem-inverse}, \ref{lem-commutability}, and \ref{lem-square}, 
the $\Xi$-equivalence class of any $(1,2)$-ladder up to shell-pairs
is regarded as an element in 
the abelian group presented by 
$$\left\langle A_{1}, A_{2}, A_{3}, A_{4}\left| 
\begin{array}{l}
A_iA_j=A_jA_i \ (1\leq i<j\leq 4), \\
A_1A_2=A_3A_4, \ A_1A_3=A_2A_4, \ A_2^2=A_3^2
\end{array}\right.
\right\rangle.$$
Since this group is isomorphic to 
$$\langle A_{1},A_{2},A_{3}\mid 
A_iA_j=A_jA_i \ (1\leq i<j\leq 3), \ 
A_{2}^2=A_{3}^2\rangle
\cong\mathbb{Z}\times\mathbb{Z}\times(\mathbb{Z}/2\mathbb{Z}),$$ 
we have the conclusion. 
\end{proof}

Similarly to a $(1,2)$-ladder, 
we define a {\it $(2,1)$-ladder} as a portion 
consisting of nonself-chords of type $(2,1)$, 
which are labeled $B_{1}^{\varepsilon}$, $B_{2}^{\varepsilon}$, 
$B_{3}^{\varepsilon}$, and $B_{4}^{\varepsilon}$ $(\varepsilon=\pm 1)$ 
as shown in Figure~\ref{labels-hat}. 
Then we have similar properties to Lemmas~\ref{lem-inverse}--\ref{lem-nonself-chord} 
as follows.

\begin{figure}[htbp]
\centering
    \vspace{1em}
    \begin{overpic}[width=11cm]{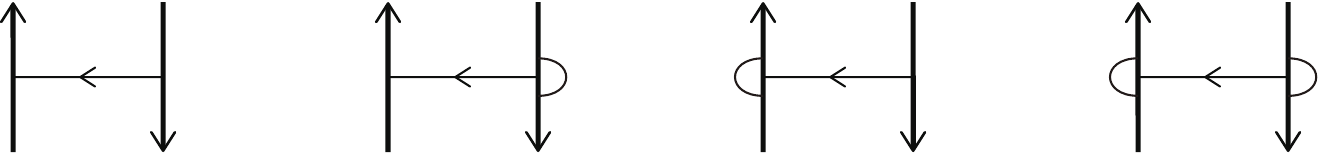}
      \put(-2,41){$C_{1}$}
      \put(34,41){$C_{2}$}
      \put(19,23){$\e$}
      \put(17,-14){$B_{1}^{\e}$}
      %%%
      \put(87,41){$C_{1}$}
      \put(123,41){$C_{2}$}
      \put(108,23){$\e$}
      \put(106,-14){$B_{2}^{\e}$}
      %%%
      \put(176,41){$C_{1}$}
      \put(212,41){$C_{2}$}
      \put(197,23){$\e$}
      \put(195,-14){$B_{3}^{\e}$}
      %%%
      \put(265,41){$C_{1}$}
      \put(301,41){$C_{2}$}
      \put(286,23){$\e$}
      \put(284,-14){$B_{4}^{\e}$}
    \end{overpic}
  \vspace{1em}
  \caption{The labels $B_{1}^{\e},B_{2}^{\e},B_{3}^{\e}$ and $B_{4}^{\e}$ of nonself-chords of type~$(2,1)$}
  \label{labels-hat}
\end{figure}

\begin{lemma}\label{lem-inverse2}
For $i\in \{1,2,3,4\}$,  
we have the $\Xi$-equivalence 
$$B_iB_i^{-1}\sim B_i^{-1}B_i\sim\emptyset$$ 
up to shell-pairs. 
\hfill$\Box$
\end{lemma}

\begin{lemma}\label{lem-commutability2}
We have the following $\Xi$-equivalence up to shell-pairs. 
\begin{enumerate}
\item[(i)] $B_{1}B_{2}\sim B_{2}B_{1}\sim B_{3}B_{4}\sim B_{4}B_{3}$.  
\item[(ii)] $B_{1}B_{3}\sim B_{3}B_{1}\sim B_{2}B_{4}\sim B_{4}B_{2}$. 
\item[(iii)] $B_{1}B_{4}\sim B_{4}B_{1}$ and $B_{2}B_{3}\sim B_{3}B_{2}$. 
\end{enumerate}
In particular, the letters $B_{1}$, $B_{2}$, $B_{3}$, and $B_{4}$ 
are mutually commutative. 
\hfill$\Box$
\end{lemma}

\begin{lemma}\label{lem-square2}
We have the $\Xi$-equivalence $B_{2}^{2}\sim B_{3}^{2}$ up to shell-pairs. 
\hfill$\Box$
\end{lemma}

\begin{lemma}\label{lem-nonself-chord2}
Any $(2,1)$-ladder is $\Xi$-equivalent to either 
$B_{1}^{b_{1}}B_{2}^{b_{2}}$ or $B_{1}^{b_{1}}B_{2}^{b_{2}}B_{3}$ for some $b_{1},b_{2}\in\mathbb{Z}$ up to  shell-pairs. 
\hfill$\Box$
\end{lemma}

For a $(1,2)$-ladder $W$ and a $(2,1)$-ladder $V$, 
we consider the ladder $WV$ obtained by stacking $W$ on top of $V$, as for example on the left-hand side of Figure \ref{pf-lem-C}. 
Then we have the following.

\begin{lemma}\label{lem-C}
We have the $\Xi$-equivalence
$A_{2}B_{2}\sim A_{3}B_{3}$ up to shell-pairs. 
\end{lemma}

\begin{proof}
This follows by the deformation 
as shown in Figure~\ref{pf-lem-C}. 
We remark that the R3-move in the figure 
is the same as the one used at the bottom 
of Figure~\ref{R3-Gauss}. 
\end{proof}

\begin{figure}[htbp]
\centering
    \begin{overpic}[width=12cm]{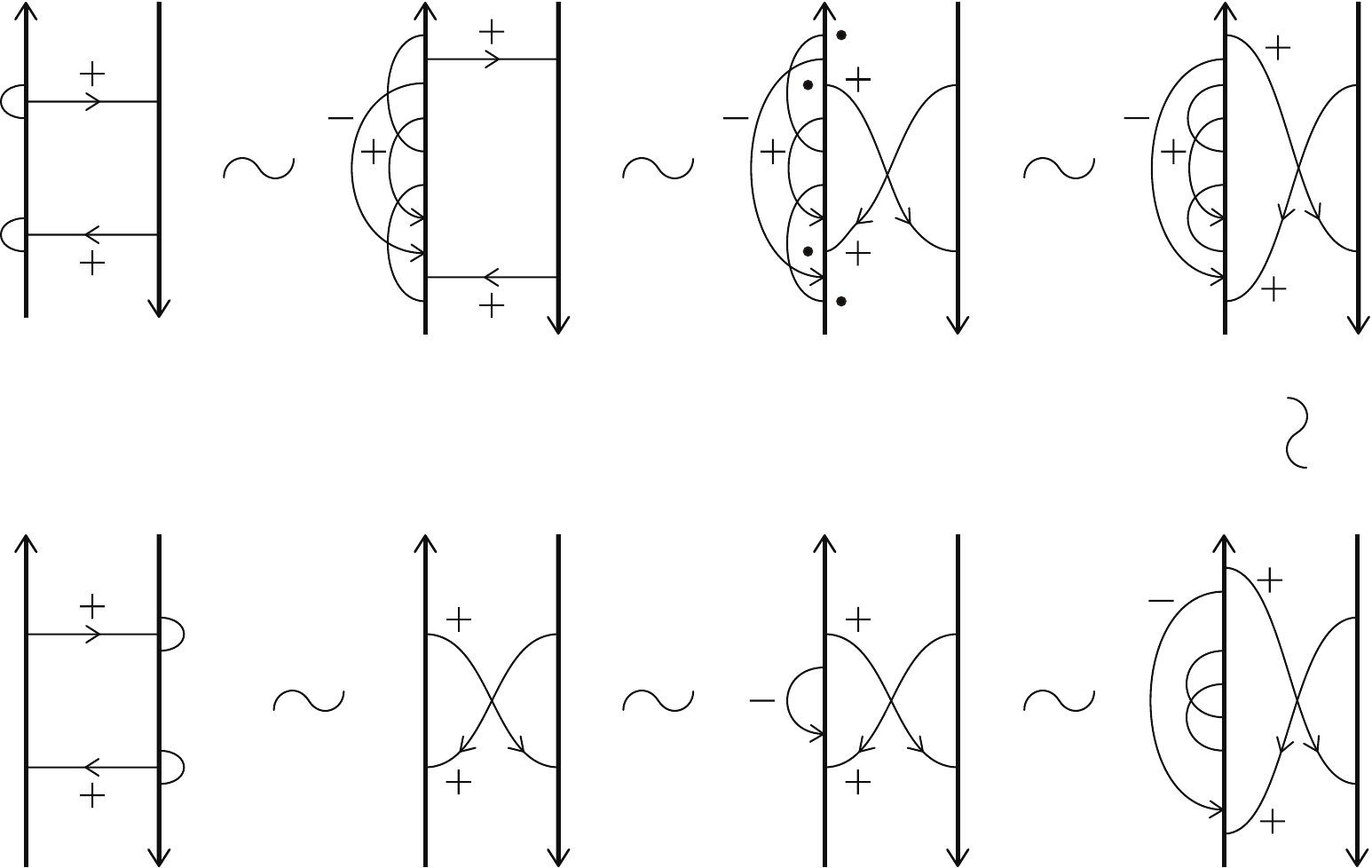}
      \put(59,180){R2}
      \put(158,180){R3}
      \put(261,180){$\Xi$}
      \put(253,161){twice}
      %%%
      \put(58,48){Lem~\ref{lem-exchange}}
      \put(158,48){R1}
      \put(247,48){Lem~\ref{lem-sliding}}
      \put(280,105){Lem~\ref{lem-3shells}}
      %%%
      \put(12,-15){$A_{2}B_{2}$}
      \put(12,117){$A_{3}B_{3}$}
    \end{overpic}
  \vspace{1em}
  \caption{Proof of Lemma~\ref{lem-C}}
  \label{pf-lem-C}
\end{figure}

\begin{proposition}\label{prop-nonself-chord}
For a $(1,2)$-ladder $W$ and a $(2,1)$-ladder $V$, 
the ladder $WV$ is $\Xi$-equivalent to either 
$$A_{1}^{a_{1}}A_{2}^{a_{2}}B_{1}^{b_{1}}B_{2}^{b_{2}} \text{ or } A_{1}^{a_{1}}A_{2}^{a_{2}}B_{1}^{b_{1}}B_{2}^{b_{2}}B_{3}$$
for some $a_{1},a_{2},b_{1},b_{2}\in\mathbb{Z}$ up to shell-pairs. 
\end{proposition}

\begin{proof}
By Lemmas~\ref{lem-nonself-chord} and \ref{lem-nonself-chord2}, 
we have 
\[
W\sim A_{1}^{a_{1}}A_{2}^{a_{2}}A_{3}^{a_{3}} \text{ and }
V\sim B_{1}^{b_{1}}B_{2}^{b_{2}}B_{3}^{b_{3}}
\] 
for some $a_{1},a_{2},b_{1},b_{2}\in\mathbb{Z}$ and $a_{3},b_{3}\in\{0,1\}$. 
For $a_{3}=0$, we have
$WV\sim 
A_{1}^{a_{1}}A_{2}^{a_{2}}B_{1}^{b_{1}}B_{2}^{b_{2}}B_{3}^{b_{3}}$ 
which gives a conclusion.

Now we consider the case $a_{3}=1$. 
It holds that $A_{3}\sim A_{2}B_{2}B_{3}^{-1}$ by Lemmas~\ref{lem-inverse2} and \ref{lem-C}. 
Since $B_{1}$, $B_{2}$, and $B_{3}$ commute mutually 
up to $\Xi$-equivalence by Lemma~\ref{lem-commutability2}, 
we have 
\[
\begin{split}
WV&\sim A_{1}^{a_{1}}A_{2}^{a_{2}}A_{3}B_{1}^{b_{1}}B_{2}^{b_{2}}B_{3}^{b_{3}} \\
&\sim A_{1}^{a_{1}}A_{2}^{a_{2}}(A_{2}B_{2}B_{3}^{-1})B_{1}^{b_{1}}B_{2}^{b_{2}}B_{3}^{b_{3}} \\
&\sim A_{1}^{a_{1}}A_{2}^{a_{2}+1}B_{1}^{b_{1}}B_{2}^{b_{2}+1}B_{3}^{b_{3}-1}.
\end{split}
\]
If $b_{3}=1$, then we have a conclusion. 
We consider the case $b_{3}=0$, that is, 
$WV\sim A_{1}^{a_{1}}A_{2}^{a_{2}+1}B_{1}^{b_{1}}B_{2}^{b_{2}+1}B_{3}^{-1}$. 
Since $B_{3}^{-1}\sim B_{2}^{-2}B_{3}$ holds 
by Lemma~\ref{lem-square2}, 
we have 
\[
\begin{split}
WV&\sim 
A_{1}^{a_{1}}A_{2}^{a_{2}+1}B_{1}^{b_{1}}B_{2}^{b_{2}+1}B_{3}^{-1} \\
&\sim
A_{1}^{a_{1}} A_{2}^{a_{2}+1}B_{1}^{b_{1}}B_{2}^{b_{2}+1}(B_{2}^{-2}B_{3}) \\
&\sim
A_{1}^{a_{1}} A_{2}^{a_{2}+1} B_{1}^{b_{1}} B_{2}^{b_{2}-1}B_{3}.
\end{split}
\]
\end{proof}

For integers $a_1,a_2,b_1,b_2,k,l\in\mathbb{Z}$, 
we denote by $$G(a_1,a_2, b_1,b_2;k,l) \text{ and }
H(a_1,a_2, b_1,b_2;k,l)$$ the Gauss diagrams 
as shown in Figure~\ref{rep-even}. 
More precisely, their ladders are 
$$A_1^{a_1}A_2^{a_2}B_1^{b_1}B_2^{b_2} \text{ and }
A_1^{a_1}A_2^{a_2}B_1^{b_1}B_2^{b_2}B_3,$$ respectively, 
and all nonself-chords in $A_1^{a_1}$ (resp. $A_2^{a_2}$, $B_1^{b_1}$, and $B_2^{b_2}$) have the sign $e(a_1)$ (resp. $e(a_2)$, $e(b_1)$, and $e(b_2)$). 
Furthermore, there are $|k|$ shell-pairs with sign $e(k)$ on $C_{1}$ 
and $|l|$ shell-pairs with sign $e(l)$ on $C_{2}$.
We denote by 
$$L(a_1,a_2, b_1,b_2;k,l) \text{ and } M(a_1,a_2, b_1,b_2;k,l)$$ 
the $2$-component virtual links 
presented by the Gauss diagrams 
$G(a_1,a_2, b_1,b_2;k,l)$ and $H(a_1,a_2, b_1,b_2;k,l)$, 
respectively.

\begin{figure}[htbp]
\centering
  \begin{overpic}[width=7.5cm]{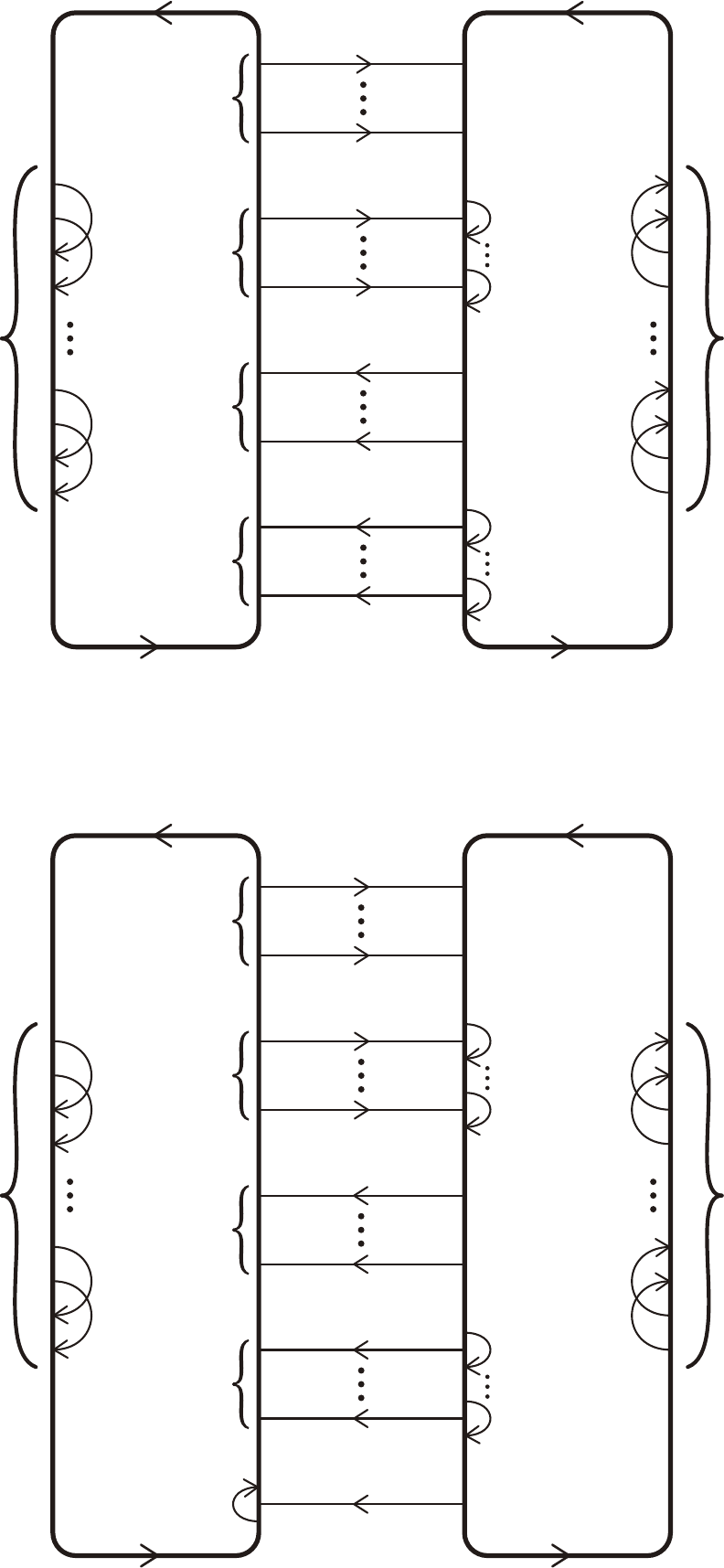}
    %%%
    \put(65,254){$G(a_{1},a_{2},b_{1},b_{2};k,l)$}
    \put(65,-14){$H(a_{1},a_{2},b_{1},b_{2};k,l)$}
    %%% upper %%%
    \put(-1,431){$C_{1}$}
    \put(203,431){$C_{2}$}
    %%%
    \put(117,448){$\gamma_{0}$}
    \put(50,431){$|a_{1}|$}
    \put(80,448){$e(a_{1})$}
    \put(80,428){$e(a_{1})$}
    %%%
    \put(50,386){$|a_{2}|$}
    \put(80,402){$e(a_{2})$}
    \put(80,382){$e(a_{2})$}
    \put(147,396){$e(a_{2})$}
    \put(147,376){$e(a_{2})$}
    %%%
    \put(51,340){$|b_{1}|$}
    \put(80,357){$e(b_{1})$}
    \put(80,337){$e(b_{1})$}
    %%%
    \put(51,295){$|b_{2}|$}
    \put(80,311){$e(b_{2})$}
    \put(80,291){$e(b_{2})$}
    \put(147,303){$e(b_{2})$}
    \put(147,283){$e(b_{2})$}
    %%%
    \put(-14,331){\rotatebox{90}{$|k|$ shell-pairs}}
    \put(29,399){$e(k)$}
    \put(29,381){$e(k)$}
    \put(29,338){$e(k)$}
    \put(29,321){$e(k)$}
    %%%
    \put(218,391){\rotatebox{-90}{$|l|$ shell-pairs}}
    \put(175,410){$e(l)$}
    \put(175,372){$e(l)$}
    \put(175,348){$e(l)$}
    \put(175,311){$e(l)$}

    %%% lower %%%
    \put(-1,188){$C_{1}$}
    \put(203,188){$C_{2}$}
    %%%
    \put(117,205){$\gamma_{0}$}
    \put(50,188){$|a_{1}|$}
    \put(80,205){$e(a_{1})$}
    \put(80,185){$e(a_{1})$}
    %%%
    \put(50,143){$|a_{2}|$}
    \put(80,159){$e(a_{2})$}
    \put(80,139){$e(a_{2})$}
    \put(147,153){$e(a_{2})$}
    \put(147,133){$e(a_{2})$}
    %%%
    \put(51,97){$|b_{1}|$}
    \put(80,114){$e(b_{1})$}
    \put(80,94){$e(b_{1})$}
    %%%
    \put(51,52){$|b_{2}|$}
    \put(80,69){$e(b_{2})$}
    \put(80,49){$e(b_{2})$}
    \put(147,62){$e(b_{2})$}
    \put(147,42){$e(b_{2})$}
    %%%
    \put(58,16){$+$}
    \put(87,22){$+$}
    %%%
    \put(-14,79){\rotatebox{90}{$|k|$ shell-pairs}}
    \put(29,146){$e(k)$}
    \put(29,128){$e(k)$}
    \put(29,86){$e(k)$}
    \put(29,67){$e(k)$}
    %%%
    \put(218,139){\rotatebox{-90}{$|l|$ shell-pairs}}
    \put(175,157){$e(l)$}
    \put(175,119){$e(l)$}
    \put(175,95.5){$e(l)$}
    \put(175,58){$e(l)$}
  \end{overpic}
  \vspace{1em}
  \caption{$G(a_{1},a_{2},b_{1},b_{2};k,l)$ and $H(a_{1},a_{2},b_{1},b_{2};k,l)$}
  \label{rep-even}
\end{figure}

\begin{proposition}\label{prop-even}
Any $2$-component virtual link $L$ is $\Xi$-equivalent to either 
$$L(a_1,a_2, b_1,b_2; k,l) \text{ or } M(a_1,a_2, b_1,b_2;k,l)$$
for some $a_1,a_2,b_1,b_2,k,l\in\mathbb{Z}$. 
\end{proposition}

\begin{proof} 
Any Gauss diagram of $L$ is 
$\Xi$-equivalent to a Gauss diagram $G_1$ 
which satisfies the conditions (i) and (ii) in Proposition~\ref{prop-selfchord}.

Let $W$ and $V$ be the $(1,2)$- and $(2,1)$-ladders of $G_1$, respectively. 
By Proposition~\ref{prop-nonself-chord}, 
the ladder $WV$ is $\Xi$-equivalent to 
either $$A_{1}^{a_{1}}A_{2}^{a_{2}}B_{1}^{b_{1}}B_{2}^{b_{2}}\text{ or }
A_{1}^{a_{1}}A_{2}^{a_{2}}B_{1}^{b_{1}}B_{2}^{b_{2}}B_{3}$$ 
with producing a finite number of shell-pairs 
for some $a_1,a_2,b_1,b_2\in\mathbb{Z}$. 
Let $G_2$ be the obtained Gauss diagram.

By Lemma~\ref{lem-sign}, 
$G_2$ is $\Xi$-equivalent to a Gauss diagram $G_3$ 
such that any shell at an endpoint of a nonself-chord 
labeled $A_2^{\e}$, $B_2^{\e}$, or $B_3$ 
has the same sign as that of the nonself-chord. 
We may produce a finite number of shell-pairs 
in this $\Xi$-equivalence.

If a circle $C_i$ $(i=1,2)$ of $G_{3}$ has a shell-pair consisting of 
positive and negative shells, 
then we delete it by an R2-move. 
Furthermore, if $C_i$ has a pair of positive and negative shell-pairs, 
then we cancel it by Lemmas~\ref{lem-sliding} and 
\ref{lem-canceling-pairs}. 
Let $G_4$ be the obtained Gauss diagram, 
where all shell-pairs in each $C_i$ $(i=1,2)$ have 
the same sign.

Recall that the orientation of a shell can be altered 
by a $\Xi$-move (without producing new shell-pairs). 
Therefore $G_4$ is $\Xi$-equivalent to a Gauss diagram $G_5$ 
such that the orientation of any shell on $C_i$ is coherent to 
that of $C_i$ $(i=1,2)$ as  in Figure~\ref{rep-even}. 
This Gauss diagram $G_5$ is coincident with 
$G(a_1,a_2, b_1,b_2; k,l)$ or $H(a_1,a_2, b_1,b_2; k,l)$ finally. 
\end{proof}

%%%%%%%%%% Even case: Invariants %%%%%%%%%
\section{Invariants of a $2$-component even virtual link}
\label{sec-even-inv}

Throughout Sections~\ref{sec-even-inv} and \ref{sec-even-proof}, 
we consider a $2$-component {\em even} virtual link $L=K_{1}\cup K_{2}$ 
and its Gauss diagram $G$. 
In this section, 
we introduce three kinds of invariants 
$J(L, K_{1})$, $J(L, K_{2})$, and $\overline{F}(L)$ of $L$, 
and establish a relationship among these invariants (Theorem~\ref{th-relation}). 

Let $\gamma$ be a self-chord on a circle $C_{i}$ of $G$. 
The endpoints of $\gamma$ divide $C_{i}$ into two arcs. 
Let $\alpha$ be one of the two arcs. 
We define the \textit{parity} of $\gamma$ to be 
the parity of the number of endpoints of self-/nonself-chords on $\alpha$. 
Since the number of nonself-chords in $G$ is even, 
the parity of $\gamma$ does not depend on a particular choice of the arc $\alpha$; 
in fact, 
the number of endpoints of self-/nonself-chords 
on $C_i$ $(i=1,2)$ is even. 
By definition, any shell is odd. 
Let $J(G,C_i)$ denote the sum of signs of all odd self-chords on $C_{i}$.

\begin{example}\label{ex-oddwrithe}
Consider the Gauss diagram $G=H(3,-2,1,1;-3,2)$ as shown in Figure~\ref{ex-inv}. 
Then the self-chords on $C_{1}$ consist of one positive shell and six negative shells, 
and those on $C_{2}$ consist of five positive shells and two negative shells. 
Therefore we have 
\[
J(G,C_{1})=-5 \text{ and } J(G,C_{2})=3. 
\] 
\end{example}

\begin{figure}[htbp]
\centering
    \begin{overpic}[width=4cm]{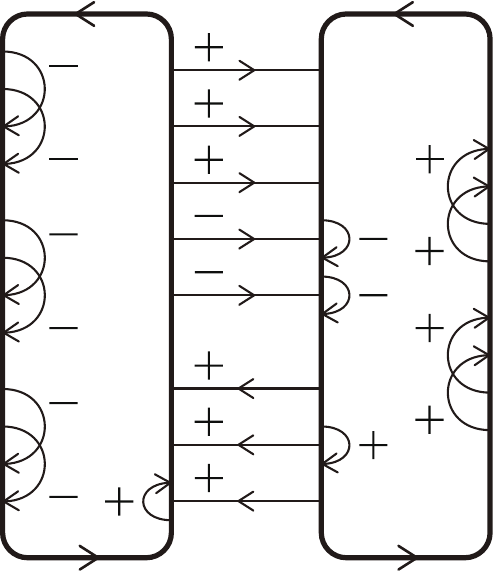}
      \put(61,120){$\gamma_{0}$}
      \put(-18,62){$C_{1}$}
      \put(119,62){$C_{2}$}
    \end{overpic}
  \caption{The Gauss diagram $H(3,-2,1,1;-3,2)$}
  \label{ex-inv}
\end{figure}

\begin{lemma}\label{lem-oddwrithe}
The integer $J(G,C_{i})$ is an invariant of the $2$-component even virtual link $L$ for $i=1,2$. 
Furthermore, it is invariant under $\Xi$-moves. 
\end{lemma}

\begin{proof}
An R1-move adds or removes an even self-chord, and preserves the parity of any other self-chords. 
An R2-move adds or deletes a pair of chords $\gamma$ and $\gamma'$ with opposite signs, 
and preserves the parity of any other self-chords. 
If $\gamma$ and $\gamma'$ are nonself-chords, then they do not contribute to $J(G,C_{i})$. 
If $\gamma$ and $\gamma'$ are self-chords on $C_{i}$, then they have the same parity and the contributions to $J(G,C_{i})$ cancel out. 
An R3-move or a $\Xi$-move preserves the sign and parity of any self-chords. 
\end{proof}

\begin{definition}\label{def-oddwrithe}
The integer $J(G,C_{i})$ is called 
the {\em odd writhe of the pair $(L,K_{i})$} $(i=1,2)$, 
and is denoted by $J(L,K_{i})$.
\end{definition}

For example, the virtual link 
$L=K_1\cup K_2=M(3,-2,1,1;-3,2)$ presented by the Gauss diagram $H(3,-2,1,1;-3,2)$ 
of Figure \ref{ex-inv} has the odd writhes  
$$J(L,K_1)=-5 \text{ and }J(L,K_2)=3.$$
We stress that the odd writhe $J(L,K_{i})$ of the pair $(L,K_{i})$ 
is different from the original odd writhe $J(K_i)$ 
of the virtual knot $K_i$ itself introduced in~\cite{K04}, 
meaning that $J(L,K_{i})$ is an invariant of $L$ rather than $K_i$.

\bigskip

Fix a nonself-chord $\gamma_0$ in $G$. 
For any other nonself-chord $\gamma$ in $G$, 
the endpoints of $\gamma_0$ and $\gamma$ on $C_{1}$ 
divide the circle $C_{1}$ into two arcs. 
Let $\alpha$ be one of the two arcs. 
Similarly, the endpoints of $\gamma_0$ and $\gamma$ on $C_{2}$ 
divide $C_{2}$ into two arcs, and let $\beta$ be one of the two arcs. 
See Figure~\ref{arcs}. 
We consider two sets of nonself-chords $\gamma$ in $G$ as follows; 
\[
S(\gamma_0)=\{\gamma\mid 
\text{the number of 
endpoints of chords on $\alpha\cup\beta$ is even}\}
\]
and 
\[
T(\gamma_0)=\{\gamma\mid 
\text{the number of 
endpoints of chords on $\alpha\cup\beta$ is odd}\}.
\]
As a convention, we set $\gamma_0\in S(\gamma_0)$. 
Since the number of nonself-chords in $G$ is even, 
the parity of the number of endpoints of chords on $\alpha\cup\beta$ 
does not depend on a particular choice of the arcs $\alpha$ and $\beta$. 
Therefore the sets $S(\gamma_0)$ and $T(\gamma_0)$ 
are well-defined for $\gamma_0$.

\begin{figure}[htbp]
\centering
    \begin{overpic}[width=3cm]{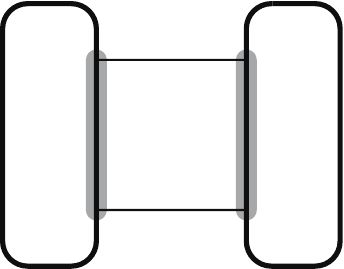}
      \put(12,31){$\alpha$}
      \put(67,31){$\beta$}
      \put(-15,31){$C_{1}$}
      \put(89,31){$C_{2}$}
      \put(40,57){$\gamma_{0}$}
      \put(40,20){$\gamma$}
    \end{overpic}
  \caption{A pair of arcs $\alpha$ and $\beta$ for nonself-chords $\gamma_{0}$ and $\gamma$}
  \label{arcs}
\end{figure}

Let $\sigma_{12}(G, \gamma_{0})$ and $\tau_{12}(G, \gamma_0)$ 
be the sums of signs of all nonself-chords of type~$(1,2)$ 
in $S(\gamma_0)$ and $T(\gamma_0)$, respectively. 
Similarly, 
let $\sigma_{21}(G, \gamma_{0})$ and $\tau_{21}(G, \gamma_0)$ 
be the sums of signs of all nonself-chords 
of type~$(2,1)$ in $S(\gamma_0)$ and $T(\gamma_0)$, respectively. 
By definition, we have 
\[
\sigma_{12}(G, \gamma_{0})+\tau_{12}(G, \gamma_{0})={\rm Lk}(K_{1},K_{2}) \text{ and } 
\sigma_{21}(G, \gamma_{0})+\tau_{21}(G, \gamma_{0})={\rm Lk}(K_{2},K_{1}). 
\]

We introduce an equivalence relation $\doteq$ 
among the elements in $\mathbb{Z}^4$ as follows. 
For two elements $(a_{1},a_{2},b_{1},b_{2})$ and $(c_{1},c_{2},d_{1},d_{2})\in\mathbb{Z}^{4}$,
we define 
$(a_1,a_2,b_1,b_2)\doteq(c_1,c_2,d_1,d_2)$ 
if and only if 
$$(a_1,a_2,b_1,b_2)=(c_1,c_2,d_1,d_2) \text{ or }
(c_2,c_1,d_2,d_1).$$
We denote by 
$[a_1,a_2,b_1,b_2]$ the equivalence class of 
$(a_1,a_2,b_1,b_2)$ under the equivalence relation 
$\doteq$. 
For a nonself-chord $\gamma_0$ in $G$, 
we put 
$$\overline{F}(G, \gamma_0)=
\bigl[\sigma_{12}(G, \gamma_0),\tau_{12}(G, \gamma_0), 
\sigma_{21}(G, \gamma_0),\tau_{21}(G, \gamma_0)\bigr]
\in\mathbb{Z}^4/\doteq.$$

\begin{example}
Consider the Gauss diagram $G=H (3, -2, 1, 1; -3, 2)$ 
given in Example~\ref{ex-oddwrithe}. 
The ladder of $G$ 
is expressed by $A_{1}^{3}A_{2}^{-2}B_{1}B_{2}B_{3}$. 
Let $\gamma_{0}$ be the top nonself-chord labeled $A_1$ 
as shown in Figure~\ref{ex-inv}. 
Then the set $S(\gamma_0)$ consists of nonself-chords labeled $A_1$ or $B_2$, 
and $T(\gamma_0)$ consists of 
those labeled $A_2^{-1}$, $B_2$, or $B_3$. 
Hence we have 
\[
\sigma_{12}(G, \gamma_{0})=3,\ 
\tau_{12}(G, \gamma_{0})=-2,\ 
\sigma_{21}(G, \gamma_{0})=1, \text{ and }
\tau_{21}(G, \gamma_{0})=2, 
\]
and it holds that 
$\overline{F}(G, \gamma_{0})=[3,-2,1,2]=[-2,3,2,1]$. 
\end{example}

\begin{lemma}\label{lem-linkingclass}
The equivalence class $\overline{F}(G, \gamma_{0})$ is 
an invariant of the $2$-component even virtual link $L$. 
Furthermore, it is invariant under $\Xi$-moves. 
\end{lemma}

\begin{proof}
We first prove that 
$\overline{F}(G, \gamma_{0})$ does not depend on a particular choice of~$\gamma_{0}$. 
Consider a nonself-chord $\gamma_{1}$ in $G$. 
In the case $\gamma_{1}\in S(\gamma_0)$, we have
\[
\sigma_{ij}(G, \gamma_{0})=\sigma_{ij}(G, \gamma_{1})\ \text{and}\ 
\tau_{ij}(G, \gamma_{0})=\tau_{ij}(G, \gamma_{1}) 
\]
for $\{i,j\}=\{1,2\}$. 
In the case $\gamma_{1}\in T(\gamma_{0})$, we have
\[
\sigma_{ij}(G, \gamma_{0})=\tau_{ij}(G, \gamma_{1})\ \text{and}\ 
\tau_{ij}(G, \gamma_{0})=\sigma_{ij}(G, \gamma_{1}) 
\]
for $\{i,j\}=\{1,2\}$. 
Therefore $\overline{F}(G, \gamma_{0})=\overline{F}(G, \gamma_{1})$ holds.

We can also prove that $\overline{F}(G,\gamma_{0})$ is invariant under an R2-move
adding a pair of nonself-chords. 
In fact, it leaves $\sigma_{ij}(G, \gamma_{0})$ and $\tau_{ij}(G, \gamma_{0})$ unchanged for $\{i,j\}=\{1,2\}$.

Now we consider a Reidemeister move R1--R3 or $\Xi$-move on $G$ generally.
By using the above two properties if necessary, we may assume that such a move does not involve $\gamma_{0}$. 
Then it can be seen that this move preserves 
$\sigma_{ij}(G, \gamma_{0})$ and $\tau_{ij}(G, \gamma_{0})$  for $\{i,j\}=\{1,2\}$. 
\end{proof}

\begin{definition}\label{def-linkingclass}
The equivalence class $\overline{F}(G, \gamma_0)\in\mathbb{Z}^4/\doteq$ 
is called the {\em reduced linking class} of $L$, 
and is denoted by $\overline{F}(L)$. 
\end{definition}

\begin{remark}
There is an invariant of $L$ 
called the {\it linking class} $F(L)$ (cf.~\cite{CG, NNS}). 
We can prove that the reduced linking class $\overline{F}(L)$ 
is obtained from $F(L)$. 
We do not give the definition of $F(L)$ in this paper, 
and leave the proof to the reader. 
\end{remark}

Now we have three kinds of invariants 
$J(L,K_1)$, $J(L,K_2)$, and $\overline{F}(L)$. 
To study a relationship among these invariants, 
we prepare the following two lemmas. 
Recall that $L(a_1,a_2,b_1,b_2;k,l)$ and 
$M(a_1,a_2,b_1,b_2;k,l)$ are the $2$-component 
virtual links introduced at the end of Section~\ref{sec-standardform}. 

\begin{lemma}\label{lem-even}
Let $a_1,a_2,b_1,b_2,k,l\in \mathbb{Z}$.
We have the following. 
\begin{enumerate}
\item[(i)] 
$L(a_{1},a_{2},b_{1},b_{2};k,l)$ is even if and only if 
\[ 
a_{1}+a_{2}+b_{1}+b_{2}\equiv 0\pmod{2}. 
\] 

\item[(ii)] 
$M(a_{1},a_{2},b_{1},b_{2};k,l)$ is even if and only if 
\[ 
a_{1}+a_{2}+b_{1}+b_{2}\equiv 1\pmod{2}. 
\] 
\end{enumerate}
\end{lemma}

\begin{proof}
Recall that a $2$-component virtual link is even 
if and only if the number of nonself-chords in its Gauss diagram is even. 
Since the numbers of nonself-chords in 
 $G(a_1,a_2,b_1,b_2;k,l)$ and $H(a_1,a_2,b_1,b_2;k,l)$ 
 are equal to 
 \[
|a_{1}|+|a_{2}|+|b_{1}|+|b_{2}| \text{ and }
|a_{1}|+|a_{2}|+|b_{1}|+|b_{2}|+1,
\] 
respectively, we have the conclusion. 
\end{proof}

\begin{lemma}\label{lem-value}
We have the following. 
\begin{enumerate}
\item[(i)] 
If $L=K_{1}\cup K_{2}=L(a_{1},a_{2},b_{1},b_{2};k,l)$ is even, 
then it holds that 
\[
J(L,K_1)=2k,\ J(L,K_2)=a_{2}+b_{2}+2l, \text{ and } 
\overline{F}(L)=[a_{1},a_{2},b_{1},b_{2}]. 
\]

\item[(ii)] 
If $L=K_{1}\cup K_{2}=M(a_{1},a_{2},b_{1},b_{2};k,l)$ is even, 
then it holds that 
\[
J(L,K_1)=2k+1,\ J(L,K_2)=a_{2}+b_{2}+2l	, \text{ and } 
\overline{F}(L)=[a_{1},a_{2},b_{1},b_{2}+1].
\]
\end{enumerate}
\end{lemma}

\begin{proof}
Since the proofs of (i) and (ii) are similar, we only prove (ii). 

We consider the Gauss diagram $G=H(a_{1},a_{2},b_{1},b_{2};k,l)$. 
Since any self-chord on each circle $C_{i}$ $(i=1,2)$ of $G$ is odd, 
it contributes to $J(G,C_i)$.  
The sum of all self-chords on $C_{1}$ is equal to $2k+1$, and that on $C_{2}$ is equal to $a_{2}+b_{2}+2l$. 
Therefore we have 
\[
J(L,K_{1})=J(G,C_{1})=2k+1
\text{ and }
J(L,K_{2})=J(G, C_{2})=a_{2}+b_{2}+2l. 
\]

Recall that the ladder of $G$ is 
$A_{1}^{a_{1}}A_{2}^{a_{2}}B_{1}^{b_{1}}B_{2}^{b_{2}}B_{3}$. 
Let $\gamma_{0}$ be the top nonself-chord in $G$ labeled $A_{1}^{e(a_{1})}$ 
as shown in Figure~\ref{rep-even}. 
Then the set $S(\gamma_{0})$ consists of 
the nonself-chords labeled $A_{1}^{e(a_{1})}$ or $B_{1}^{e(b_{1})}$, and 
$T(\gamma_{0})$ consists of those labeled 
$A_{2}^{e(a_{2})}$, $B_{2}^{e(b_{2})}$, or $B_{3}$. 
Since the sum of signs of all nonself-chords labeled $A_{1}^{e(a_{1})}$ (resp. $B_{1}^{e(b_{1})}$) is equal to $a_{1}$ (resp. $b_{1}$), 
we have 
\[
\sigma_{12}(G, \gamma_{0})=a_{1} \text{ and } \sigma_{21}(G, \gamma_{0})=b_{1}.
\] 
Similarly, since the sum of signs of all nonself-chords labeled $A_{2}^{e(a_{2})}$ (resp. $B_{2}^{e(b_{2})}$ or $B_{3}$) is equal to $a_2$ (resp. $b_2+1$),
we have 
\[
\tau_{12}(G,\gamma_{0})=a_{2} \text{ and } \tau_{21}(G,\gamma_{0})=b_{2}+1.
\] 
Therefore $\overline{F}(L)
=\overline{F}(G,\gamma_{0})=[a_{1},a_{2},b_{1},b_{2}+1]$ holds.
\end{proof}

We establish a relationship among the invariants 
$J(L,K_1),J(L,K_2)$, and $\overline{F}(L)$ 
of a $2$-component even virtual link $L=K_{1}\cup K_{2}$ as follows. 

\begin{theorem}\label{th-relation} 
If $\overline{F}(L)=[a_{1},a_{2},b_{1},b_{2}]$, 
then it holds that 
\[
J(L, K_{1})+J(L, K_{2})\equiv a_{1}+b_{1}\equiv a_{2}+b_{2}\pmod{2}. 
\]
\end{theorem}

\begin{proof}
By Proposition~\ref{prop-even}, 
$L=K_{1}\cup K_{2}$ is $\Xi$-equivalent to a $2$-component even virtual link 
\[
L(c_{1},c_{2},d_{1},d_{2};k,l) \text{ or } M(c_{1},c_{2},d_{1},d_{2};k,l)
\]
for some $c_{1},c_{2},d_{1},d_{2},k,l\in\mathbb{Z}$. 
We only prove the result in the case where $L$ is $\Xi$-equivalent to $M(c_{1},c_{2},d_{1},d_{2};k,l)$. 
The other case is shown similarly.

By Lemmas~\ref{lem-linkingclass} and \ref{lem-value}(ii), 
we have 
\[
\overline{F}(L)=[c_{1},c_{2},d_{1},d_{2}+1]=[a_{1},a_{2},b_{1},b_{2}]. 
\]
Therefore it holds that 
\[
(c_{1},c_{2},d_{1},d_{2}+1)=(a_{1},a_{2},b_{1},b_{2}) \text{ or } 
(a_{2},a_{1},b_{2},b_{1}). 
\]

In the case $(c_{1},c_{2},d_{1},d_{2}+1)=(a_{1},a_{2},b_{1},b_{2})$, 
since $L$ is $\Xi$-equivalent to 
\[
M(c_1,c_2,d_1,d_2;k,l)=M(a_{1},a_{2},b_{1},b_{2}-1;k,l),\]
it holds that 
\[
J(L, K_{1})+J(L, K_{2})=(2k+1)+(a_{2}+b_{2}-1+2l)\equiv a_{2}+b_{2}\pmod{2}
\] 
by Lemma~\ref{lem-oddwrithe} and \ref{lem-value}(ii). 
Since $M(a_{1},a_{2},b_{1},b_{2}-1;k,l)$ is even, 
Lemma~\ref{lem-even}(ii) gives  
\[
a_{1}+b_{1}\equiv a_{2}+b_{2}\pmod{2}.
\] 
Therefore we have 
\[
J(L,K_1)+J(L,K_2)\equiv a_{2}+b_{2}\equiv a_{1}+b_{1}\pmod{2}. 
\]

In the case $(c_{1},c_{2},d_{1},d_{2}+1)=(a_{2},a_{1},b_{2},b_{1})$, 
$L$ is $\Xi$-equivalent to 
\[
M(c_1,c_2,d_1,d_2;k,l)=M(a_{2},a_{1},b_{2},b_{1}-1;k,l).\]
Similarly to the first case, we have 
\[
\begin{split}
J(L, K_{1})+J(L, K_{2})&=(2k+1)+(a_{1}+b_{1}-1+2l)\\
&\equiv a_{1}+b_{1}\equiv a_2+b_2 \pmod{2}.
\end{split}
\]
\end{proof}

The relationship in Theorem~\ref{th-relation} is considered as 
a necessary condition for integers to be the reduced linking class of 
a $2$-component even virtual link. 
It is also a sufficient condition as follows.

\begin{proposition}\label{prop-relation} 
For any integers $a_{1},a_{2},b_{1}$ and $b_{2}$ 
with $a_{1}+b_{1}\equiv a_{2}+b_{2}\pmod{2}$, 
there exists a $2$-component even virtual link $L=K_{1}\cup K_{2}$ such that 
  \begin{enumerate}
  \item[(i)] $\overline{F}(L)=[a_{1},a_{2},b_{1},b_{2}]$ and 
  \item[(ii)] $J(L,K_1)+J(L,K_2)\equiv a_{1}+b_{1}
  \equiv a_2+b_2\pmod{2}$. 
  \end{enumerate}
\end{proposition} 

\begin{proof}
Let $L=K_{1}\cup K_{2}$ be the $2$-component virtual link 
$L(a_{1},a_{2},b_{1},b_{2};k,l)$ for some $k,l\in\mathbb{Z}$. 
Then $L$ is even by Lemma~\ref{lem-even}(i). 
Furthermore, it holds that 
$$\overline{F}(L)=[a_1,a_2,b_1,b_2]
\text{ and }
J(L,K_1)+J(L,K_2)=a_2+b_2+2k+2l$$ 
by Lemma~\ref{lem-value}(i). 
Therefore $L$ satisfies (i) and (ii). 
\end{proof}

We remark that the $2$-component even virtual link 
$M(a_{1},a_{2},b_{1},b_{2}-1;k,l)$ 
also satisfies (i) and (ii) in Proposition~\ref{prop-relation}.

%%%%%%%%%% Even case: Proof %%%%%%%%%
\section{Proof of Theorem~\ref{th-even}}\label{sec-even-proof}

The set of representatives of $2$-component virtual links 
under $\Xi$-equivalence 
given in Proposition~\ref{prop-even} 
is not complete; that is, there are $\Xi$-equivalent pairs of virtual links 
in the sets of $L(a_1,a_2,b_1,b_2;k,l)$'s and 
$M(a_1,a_2,b_1,b_2;k,l)$'s. 
For example, it is easily seen using the deformation (2) in Lemma~\ref{lem-exchange}, that $L(1,0,0,1;0,0)$ and $L(0,1,1,0;0,0)$ are $\Xi$-equivalent. 
Generally we have the following.

\begin{proposition}\label{prop-rep-even} 
We have the following. 
\begin{enumerate}
\item[(i)] 
If $L(a_{1},a_{2},b_{1},b_{2};k,l)$ is even, 
then it is $\Xi$-equivalent to 
\[
L\bigl(a_{2},a_{1},b_{2},b_{1};k,l+\tfrac{1}{2}(-a_{1}+a_{2}-b_{1}+b_{2})\bigr).
\]

\item[(ii)] 
If $M(a_{1},a_{2},b_{1},b_{2};k,l)$ is even, 
then 
it is $\Xi$-equivalent to 
\[
M\bigl(a_{2},a_{1},b_{2}+1,b_{1}-1;k,l+\tfrac{1}{2}(-a_{1}+a_{2}-b_{1}+b_{2}+1)\bigr).
\]
\end{enumerate}
\end{proposition} 

To prove this proposition, 
we prepare the following lemma. 

\begin{lemma}\label{lem-b4}
We have the $\Xi$-equivalence 
$B_4\sim B_1B_2^{-1}B_3$ 
up to shell-pairs. 
\end{lemma}

\begin{proof}
Since it holds that $B_{4}\sim B_{1}B_{2}B_{3}^{-1}$ 
by $B_{1}B_{2}\sim B_{4}B_{3}$ in Lemma~\ref{lem-commutability2}(i) 
and $B_{3}^{-1}\sim B_{2}^{-2}B_{3}$ by Lemma~\ref{lem-square2}, 
we have 
\[
B_{4}\sim B_{1}B_{2}B_{3}^{-1}\sim B_{1}B_{2}(B_{2}^{-2}B_{3})
\sim B_{1}B_{2}^{-1}B_{3}. 
\]
\end{proof}

\begin{proof}[Proof of {\rm Proposition~\ref{prop-rep-even}}] 
Since the proofs of (i) and (ii) are similar, we only prove (ii) 
by giving the $\Xi$-equivalence of the Gauss diagrams 
\[
H(a_{1},a_{2},b_{1},b_{2};k,l)\sim 
H\bigl(a_{2},a_{1},b_{2}+1,b_{1}-1;k,l+\tfrac{1}{2}(-a_{1}+a_{2}-b_{1}+b_{2}+1)\bigr).
\] 
We remark that $a_1+a_2+b_1+b_2$ is odd 
by Lemma~\ref{lem-even}(ii).

Put $G=H(a_{1},a_{2},b_{1},b_{2};k,l)$. 
Let $G_{1}$ be the Gauss diagram obtained from $G$ 
by adding a free chord $\gamma$ on $C_2$ 
at the top of the ladder of $G$ by an R1-move 
as shown in Figure~\ref{pf-prop-rep-even}(A). 
Let $P$ be the terminal endpoint of $\gamma$. 
We will slide $P$ along the vertical line of the ladder of $G_{1}$
from top to bottom as follows.

\begin{figure}[htbp]
\centering
\vspace{1em}
  \begin{overpic}[width=11cm]{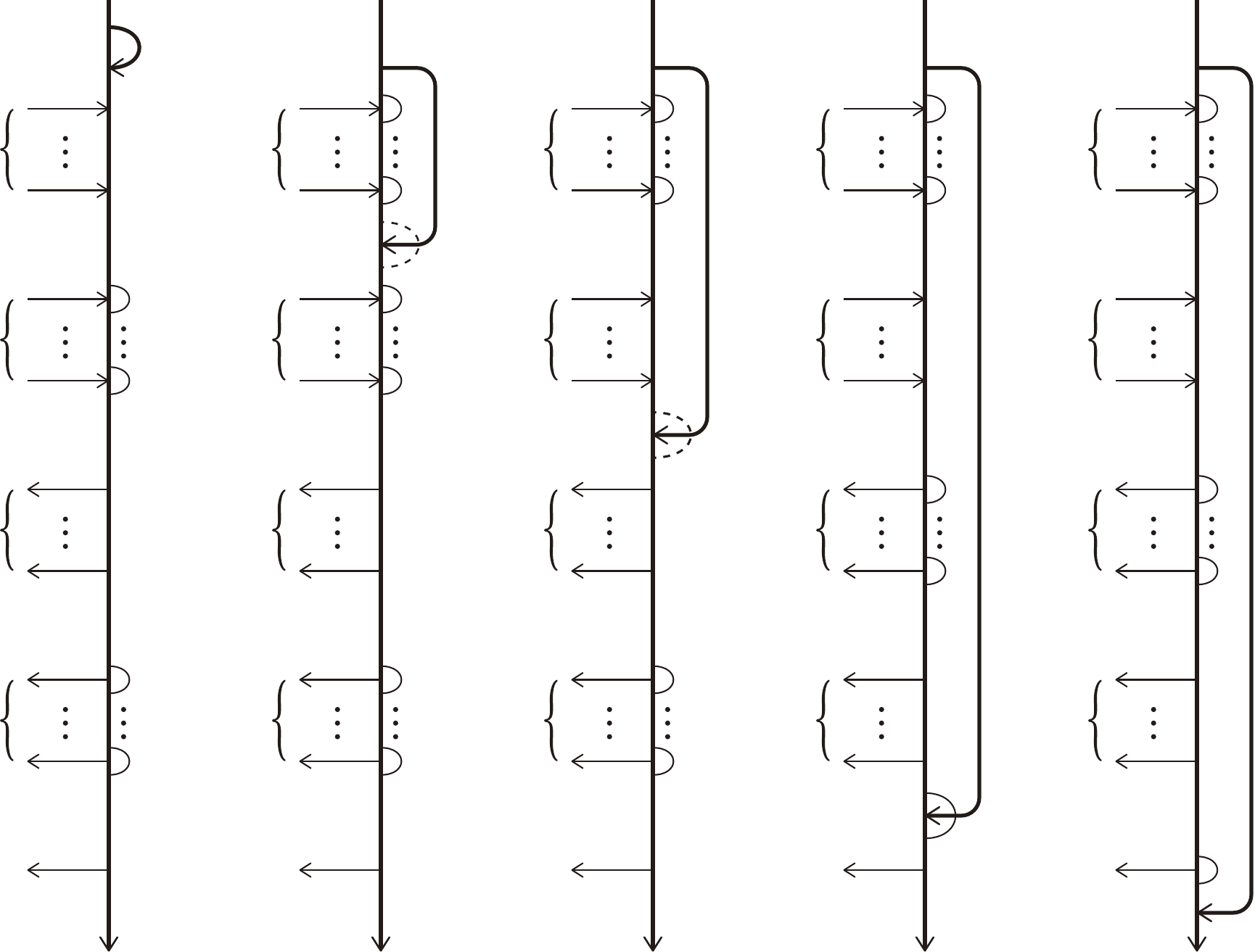}
      \put(9,-15){(A)}
      \put(77,-15){(B)}
      \put(145,-15){(C)}
      \put(213,-15){(D)}
      \put(280,-15){(E)}
      %%% (A)
      \put(21,242){$C_{2}$}
      \put(38,222){$\gamma$}
      \put(17,217){$P$}
      \put(-13,216){$A_{1}^{e(a_{1})}$}
      \put(-8,169){$A_{2}^{e(a_{2})}$}
      \put(-8,121){$B_{1}^{e(b_{1})}$}
      \put(-8,74){$B_{2}^{e(b_{2})}$}
      \put(-5,25.5){$B_{3}$}
      \put(-17,198){$|a_{1}|$}
      \put(-17,150){$|a_{2}|$}
      \put(-16,102){$|b_{1}|$}
      \put(-16,55){$|b_{2}|$}
      %%%%% (B)
      \put(90,242){$C_{2}$}
      \put(99.5,227){$\gamma$}
      \put(85,172.5){$P$}
      \put(60,216){$A_{2}^{e(a_{1})}$}
      \put(55,168){$A_{2}^{e(a_{2})}$}
      \put(60,121){$B_{1}^{e(b_{1})}$}
      \put(60,74){$B_{2}^{e(b_{2})}$}
      \put(63,25.5){$B_{3}$}
      \put(51,198){$|a_{1}|$}
      \put(51,150){$|a_{2}|$}
      \put(52,102){$|b_{1}|$}
      \put(52,55){$|b_{2}|$}
      %%%%% (C)
      \put(157,242){$C_{2}$}
      \put(166,226){$\gamma$}
      \put(153,125){$P$}
      \put(128,216){$A_{2}^{e(a_{1})}$}
      \put(128,168){$A_{1}^{e(a_{2})}$}
      \put(123,121){$B_{1}^{e(b_{1})}$}
      \put(128,74){$B_{2}^{e(b_{2})}$}
      \put(131,25.5){$B_{3}$}
      \put(119,198){$|a_{1}|$}
      \put(119,150){$|a_{2}|$}
      \put(120,102){$|b_{1}|$}
      \put(120,55){$|b_{2}|$}
      %%%%% (D)
      \put(225,242){$C_{2}$}
      \put(234,226){$\gamma$}
      \put(220.5,30){$P$}
      \put(196,216){$A_{2}^{e(a_{1})}$}
      \put(196,168){$A_{1}^{e(a_{2})}$}
      \put(196,121){$B_{2}^{e(b_{1})}$}
      \put(196,74){$B_{1}^{e(b_{2})}$}
      \put(199,25.5){$B_{3}$}
      \put(187,198){$|a_{1}|$}
      \put(187,150){$|a_{2}|$}
      \put(188,102){$|b_{1}|$}
      \put(188,55){$|b_{2}|$}
%%%%% (E)
      \put(293,242){$C_{2}$}
      \put(302,226){$\gamma$}
      \put(288,6){$P$}
      \put(264,216){$A_{2}^{e(a_{1})}$}
      \put(264,168){$A_{1}^{e(a_{2})}$}
      \put(264,121){$B_{2}^{e(b_{1})}$}
      \put(264,74){$B_{1}^{e(b_{2})}$}
      \put(267,25.5){$B_{4}$}
      \put(255,198){$|a_{1}|$}
      \put(255,150){$|a_{2}|$}
      \put(256,102){$|b_{1}|$}
      \put(256,55){$|b_{2}|$}
  \end{overpic}
\vspace{1em}
\caption{Proof of Proposition~\ref{prop-rep-even}}
\label{pf-prop-rep-even}
\end{figure}

First we exchange the positions of $P$ and 
$|a_1|$ terminal endpoints of the nonself-chords labeled $A_1^{e(a_1)}$ 
by using the deformations (1) and (2) in Lemma~\ref{lem-exchange} alternately. 
The deformations are similar to those used in the proof of Proposition~\ref{prop-even}. 
As a result, the terminal endpoint of each nonself-chord labeled $A_1^{e(a_1)}$ 
gets a shell so that the label of the chord turns into $A_2^{e(a_1)}$. 
We remark that $P$ has (resp. does not have) a shell after this deformation 
if $a_1$ is odd (resp. even). 
See Figure~\ref{pf-prop-rep-even}(B), where the potential shell at $P$ depending on the parity of $a_1$ is indicated by a dashed arc.

Next we exchange the positions of $P$ and 
$|a_2|$ terminal endpoints of nonself-chords labeled $A_2^{e(a_2)}$ 
by using the deformations (2) and (3) in Lemma~\ref{lem-exchange} alternately. 
As a result, the terminal endpoint of each nonself-chord labeled 
$A_2^{e(a_2)}$ loses its shell, 
so that the label of the chord turns into 
$A_1^{e(a_2)}$. 
Furthermore $P$ has (resp. does not have) a shell 
if $a_1+a_2$ is odd (resp. even). 
See Figure~\ref{pf-prop-rep-even}(C).

Similarly, we exchange the positions of $P$ and 
$|b_1|$ initial endpoints of the nonself-chords 
labeled $B_1^{e(b_1)}$, 
and then the positions of $P$ and 
$|b_2|$ initial endpoints of the nonself-chords 
labeled $B_2^{e(b_2)}$. 
As a result, 
the label of each nonself-chord labeled $B_{1}^{e(b_{1})}$ (resp. $B_{2}^{e(b_{2})}$) 
turns into 
$B_2^{e(b_1)}$ (resp. $B_1^{e(b_2)}$). 
Since $a_1+a_2+b_1+b_2$ is odd, 
$P$ has a shell at this point. 
See Figure~\ref{pf-prop-rep-even}(D).

Finally we exchange the positions of $P$ and 
the initial endpoint of the nonself-chord labeled $B_3$ 
by the deformation (2) in Lemma~\ref{lem-exchange} 
so that the label of the chord turns into $B_4$. 
We remark that $P$ does not have a shell. 
See Figure~\ref{pf-prop-rep-even}(E). 

Since the positions of $P$ and any shell-pair on $C_2$ can be exchanged by Lemma~\ref{lem-sliding}, 
we move $P$ next to the initial endpoint of $\gamma$, 
and then remove $\gamma$ by an R1-move. 
Let $G_2$ be the obtained Gauss diagram 
with the ladder 
$A_{2}^{a_{1}}A_{1}^{a_{2}}B_{2}^{b_{1}}B_{1}^{b_{2}}B_{4}$. 
By Lemma~\ref{lem-b4}, followed by Lemmas \ref{lem-commutability}, \ref{lem-inverse2}, and \ref{lem-commutability2}, we have 
$$A_{2}^{a_{1}}A_{1}^{a_{2}}B_{2}^{b_{1}}B_{1}^{b_{2}}B_{4}
\sim
A_{2}^{a_{1}}A_{1}^{a_{2}}B_{2}^{b_{1}}B_{1}^{b_{2}}(B_1B_2^{-1}B_3)
\sim
A_{1}^{a_{2}}A_{2}^{a_{1}}B_{1}^{b_{2}+1}B_{2}^{b_{1}-1}B_{3} 
$$
up to shell-pairs on $C_2$. 
Therefore $G_2$ is $\Xi$-equivalent to 
$$G_3=H(a_2,a_1,b_2+1,b_1-1;k,l')$$ 
for some $l'\in\mathbb{Z}$.

By Lemma~\ref{lem-value}(ii), 
we have 
$$J(G,C_2)=a_2+b_2+2l\text{ and }
J(G_3,C_2)=a_1+(b_1-1)+2l'.$$
Since $G$ and $G_3$ are $\Xi$-equivalent, 
it follows by Lemma~\ref{lem-oddwrithe} that 
$$l'=l+\frac{1}{2}(-a_{1}+a_{2}-b_{1}+b_{2}+1).$$
\end{proof}

\begin{proof}[Proof of {\rm Theorem~\ref{th-even}}]
\underline{${\rm (i)}\Rightarrow {\rm (ii)}$.}~
This follows from Lemmas~\ref{lem-oddwrithe} and~\ref{lem-linkingclass} directly. 

\underline{${\rm (ii)}\Rightarrow {\rm (i)}$.}~ 
By Proposition~\ref{prop-even}, 
$L=K_{1}\cup K_{2}$ is $\Xi$-equivalent to either 
$$L(a_1,a_2, b_1,b_2; k,l) \text{ or } M(a_1,a_2, b_1,b_2;k,l)$$
for some $a_1,a_2,b_1,b_2,k,l\in\mathbb{Z}$. 
We only prove the result in the case where 
$L$ is $\Xi$-equivalent to $M(a_1,a_2, b_1,b_2; k,l)$. 
The other case is shown similarly.

Since the odd writhe $J(L,K_1)$ of the pair $(L,K_1)$ is 
odd by Lemma~\ref{lem-value}(ii), 
so is $J(L',K_1')$ by the assumption $J(L, K_{1})=J(L', K_{1}')$. 
Therefore $L'$ is $\Xi$-equivalent to 
$M(a'_1,a'_2, b'_1,b'_2;k',l')$
for some $a'_1,a'_2,b'_1,b'_2,k',l'\in\mathbb{Z}$ 
by Proposition~\ref{prop-even} and Lemma~\ref{lem-value}(ii). 
Then it holds that 
\[
2k+1=2k'+1, \ 
a_2+b_2+2l=a_2'+b_2'+2l', \text{ and } 
[a_1,a_2,b_1,b_2+1]=[a_1',a_2',b_1',b_2'+1]
\]
by assumption and Lemma~\ref{lem-value}(ii).

We have $k=k'$ by the first equation above, and 
\[
(a_{1},a_{2},b_{1},b_{2}+1)=(a'_{1},a'_{2},b'_{1},b'_{2}+1) \text{ or }
(a'_{2},a'_{1},b'_{2}+1,b'_{1}) 
\]
by the third equation. 
In the case $(a_{1},a_{2},b_{1},b_{2}+1)=(a'_{1},a'_{2},b'_{1},b'_{2}+1)$, 
we have $l=l'$ by the second equation 
to obtain 
$$M(a_{1},a_{2},b_{1},b_{2};k,l)=M(a'_{1},a'_{2},b'_{1},b'_{2};k',l').$$
Therefore $L$ is $\Xi$-equivalent to $L'$. 
In the case $(a_{1},a_{2},b_{1},b_{2}+1)=(a'_{2},a'_{1},b'_{2}+1,b'_{1})$, 
we have 
$l=l'+\tfrac{1}{2}(-a'_{1}+a'_{2}-b'_{1}+b'_{2}+1)$ 
by the second equation to obtain 
$$M(a_{1},a_{2},b_{1},b_{2};k,l)
=M\bigl(a'_{2},a'_{1},b'_{2}+1,b'_{1}-1;k',l'+\tfrac{1}{2}(-a'_{1}+a'_{2}-b'_{1}+b'_{2}+1)\bigr).$$
Since this link is $\Xi$-equivalent to 
$M(a_1',a_2',b_1',b_2;k',l')$ by Proposition~\ref{prop-rep-even}(ii), 
$L$ is $\Xi$-equivalent to $L'$. 
\end{proof}

We consider the following subsets of 
$2$-component even virtual links; 
\[
\begin{split}
X_1&=\{L(a_1,a_2,b_1,b_2;k,l)\mid 
a_1+a_2+b_1+b_2\equiv 0 \ ({\rm mod}~2) \ {\rm and} \ a_1<a_2 \}, \\
X_2&=
\{L(a_1,a_1,b_1,b_2;k,l)\mid 
b_1+b_2\equiv 0 \ ({\rm mod}~2) \ {\rm and} \ b_1\leq b_2\}, \\
Y_1&=
\{M(a_1,a_2,b_1,b_2;k,l)\mid 
a_1+a_2+b_1+b_2\equiv 1 \ ({\rm mod}~2) \ {\rm and} \ a_1<a_2 \}, \mbox{ and}\\
Y_2&=
\{M(a_1,a_1,b_1,b_2;k,l)\mid 
b_1+b_2\equiv 1 \ ({\rm mod}~2) \ {\rm and} \ b_1\leq b_2+1\}. 
\end{split}\]

\begin{corollary}
The sets $X_1$, $X_2$, $Y_1$ and $Y_2$ 
satisfy the following. 
\begin{enumerate}
\item[(i)] 
The sets $X_1$, $X_2$, $Y_1$, and $Y_2$ are 
mutually disjoint. 
\item[(ii)]
There is no distinct pair of $2$-component even virtual links 
in $X_1$, $X_2$, 
$Y_1$, or $Y_2$ which are $\Xi$-equivalent. 
\item[(iii)] 
The disjoint union $X_1\sqcup X_2\sqcup Y_1\sqcup Y_2$ 
is a complete representative system 
of the $\Xi$-equivalence classes of 
$2$-component even virtual links. 
\end{enumerate}
\end{corollary} 

\begin{proof}
The odd writhe $J(L,K_1)$ in Lemma~\ref{lem-value} 
induces that $X_1\cup X_2$ and $Y_1\cup Y_2$ are disjoint. 
Furthermore, the reduced linking class $\overline{F}(L)$ in the lemma induces that 
 $X_1\cap X_2=\emptyset$ and $Y_1\cap Y_2=\emptyset$. 

(ii) 
Assume that two virtual links 
$L=K_{1}\cup K_{2}=L(a_1,a_2,b_1,b_2;k,l)$ and $L'=K'_{1}\cup K'_{2}=L(a_1',a_2',b_1',b_2';k',l')$ in $X_1$ 
are $\Xi$-equivalent. 
It follows from $a_1<a_2$, $a_1'<a_2'$, and $\overline{F}(L)=\overline{F}(L')$ that 
\[a_1=a_1', \ a_2=a_2', \ b_1=b_1', \text{ and }
b_2=b_2'.\] 
Then we have  $k=k'$ and $l=l'$ by $J(L,K_1)=J(L',K'_{1})$ and $J(L,K_2)=J(L',K'_{2})$. 
Since $L$ and $L'$ coincide, 
there is no distinct pair of $2$-component even virtual links in $X_1$ 
which are $\Xi$-equivalent.

Assume that two virtual links 
$L=K_{1}\cup K_{2}=L(a_1,a_1,b_1,b_2;k,l)$ and $L'=K'_{1}\cup K'_{2}=L(a_1',a_1',b_1',b_2';k',l')$ in $X_2$ 
are $\Xi$-equivalent. 
Since $b_1\leq b_2$, $b_1'\leq b_2'$, and $\overline{F}(L)=\overline{F}(L')$ hold, 
we have $b_1=b_1'$ and $b_2=b_2'$. 
Then we have $k=k'$ and $l=l'$ by $J(L,K_1)=J(L',K'_{1})$ and $J(L,K_2)=J(L',K'_{2})$. 
Therefore there is no distinct pair of $2$-component even virtual links in $X_2$ 
which are $\Xi$-equivalent. 
Similarly we can prove that this is the case 
for $Y_1$ and $Y_2$. 

(iii) 
We will prove that any $2$-component even virtual link $L$ is $\Xi$-equivalent to some virtual link 
belonging to $X_1\sqcup X_2\sqcup Y_1\sqcup Y_2$. 
By Proposition~\ref{prop-even}, $L$ can be written in the form 
\[
L(a_1,a_2,b_1,b_2;k,l) \text{ or } M(a_1,a_2,b_1,b_2;k,l)
\]
for some $a_{1},a_{2},b_{1},b_{2},k,l\in\mathbb{Z}$. 

In the case $L=L(a_1,a_2,b_1,b_2;k,l)$,  
it follows from Lemma~\ref{lem-even}(i) that 
\[a_1+a_2+b_1+b_2\equiv 0 \pmod{2}.\] 
By Proposition~\ref{prop-rep-even}(i), 
we may assume that $L$ satisfies $a_1\leq a_2$. 
For $a_1<a_2$, we have $L\in X_1$. 
For $a_1=a_2$, we have $b_1+b_2\equiv 0\pmod{2}$. 
Furthermore, we may assume that $b_1\leq b_2$ 
by Proposition~\ref{prop-rep-even}(i). 
Then it holds that $L\in X_2$. 

In the case $L=M(a_1,a_2,b_1,b_2;k,l)$, 
we can similarly prove that 
$L$ is $\Xi$-equivalent to some virtual link 
belonging to $Y_1$ or $Y_2$. 
\end{proof}

%%%%%%%%%% References %%%%%%%%%%

%%%%%%%%%%%%%%%%%%%%%%%%%%%%%%%%%%%%%%%%%%%%%%%%%%%%%%%
\end{document}